\documentclass[11pt,a4]{amsart}

\usepackage{comment}
\usepackage{amsmath,amsthm,amssymb,graphicx,fancyhdr}
\usepackage{color}
\usepackage{enumerate}
\usepackage[titletoc,title]{appendix}
\usepackage{soul}

\newtheorem{theorem}{Theorem}[section]

\newtheorem{lemma}{Lemma}[section]

\newtheorem{remark}{Remark}[section]
\newtheorem{proposition}{Proposition}[section]
   \newtheoremstyle{example}{\topsep}{\topsep}%
     {}
     {}
     {\bfseries}
     {}
     {\newline}
     {\thmname{#1}\thmnumber{ #2}\thmnote{ #3}}

   \theoremstyle{example}

\newcommand{\langlerangle}[1]{\left<#1\right>}
\newcommand{\spnu}[2]{\langlerangle{#1}_{#2}}
\newcommand{\B}[1]{{\mathbf{B}}(#1)}
\newcommand{\D}{\mathcal{D}}

\newcommand{\Prob}{\mathbb{P}}
\newcommand{\C}{\mathbb{C}}
\newcommand{\Poly}{\mathcal{P}}

\newcommand{\E}{\mathbb{E}}

\newcommand{\e}{\mathbf{e}}

\newcommand{\R}{{\mathbb{R}}}

\newcommand{\J}{J_{\alpha}}
\newcommand{\N}{\mathbb{N}}

\newcommand{\calH}{\mathcal{H}}

\newcommand{\Fou}{\mathcal{F}}

\newcommand{\Filt}{\mathbf{F}}
\newcommand{\frakp}{\mathfrak{p}}
\newcommand{\frakq}{\mathfrak{q}}

\newcommand{\hatJ}{\widehat{\mathcal{J}}_{\alpha}}
\newcommand{\hatP}{\widehat{P}}

\newcommand{\hatX}{\widehat{X}}

\newcommand{\hatZ}{\widehat{Z}}
\newcommand{\hattau}{\hat{\tau}}
\newcommand{\hatLambda}{\widehat{\Lambda}_{\alpha}}
\newcommand{\hha}{\widehat{\mathcal{H}}_{\alpha}}
\newcommand{\ha}{{\mathcal{H}}_{\alpha}}

\newcommand{\Ran}{\mathrm{Ran}}
\newcommand{\Ker}{\mathrm{Ker}}
\newcommand{\rmdq}{{\rm{d}}_q}
\newcommand{\rmdtau}{{\rm{d}}_{\tau}}
\newcommand{\pa}{{\pi}_{\alpha}}

\newcommand{\calJ}{\mathcal{J}_{\alpha}}
\newcommand{\calX}{\mathcal{X}}

\newcommand{\bo}{{\rm{O}}}
\renewcommand{\so}{{\rm{o}}}

\newcommand{\rmH}{{\rm{H}}}
\newcommand{\Lp}{{\rm{L}}^2(\R_+)}
\newcommand{\Lr}{{\rm{L}}^2(\R)}
\newcommand{\Lm}{{\rm{L}}^2(m)}
\newcommand{\eag}{\overline{\e}_{\kappa,\eta}}
\newcommand{\Leag}{{\rm{L}}^2\left(\eag\right)}
\newcommand{\Lnu}{{\rm{L}}^2(\nu)}
\newcommand{\Cd}{{\rm{C}}}
\newcommand{\Bb}{{\rm{B}}_b(\R_+)}

\newcommand{\Co}{\Cd_0(\R_+)}
\newcommand{\Ci}{\Cd^{\infty}(\R_+)}
\newcommand{\Ctwo}{\Cd^2_0(\R_+)}
\newcommand{\Ctwor}{\Cd^2_0(\R)}

\newcommand{\adjLambda}{\widehat{\Lambda}_{\alpha}}

\newcommand{\M}{\mathcal{M}}

\newcommand{\La}{\Lambda_{\alpha}}
\newcommand{\Xa}{\mathbf{X}_{\alpha}}
\newcommand{\eqindist}{\,{\buildrel d \over =}\,}

\newcommand{\siminfy}{\,{\buildrel \infty \over \sim}\,}
\newcommand{\sipminfy}{\,{\buildrel \pm \infty \over \sim}\,}
\newcommand{\eqinfy}{\,{\buildrel \infty \over =}\,}
\newcommand{\eqpinfy}{\,{\buildrel \pm \infty \over =}\,}
\newcommand{\eqzero}{\,{\buildrel 0 \over =}\,}

\newcommand{\alp}{\kappa}

\numberwithin{equation}{section}
\author{P. Patie}\thanks{ This work was partially  supported by ARC IAPAS, a fund of the Communaut\'ee francaise de Belgique. The authors are indebted to an anonymous referee who pointed out  that one of the pathwise descriptions of the dual process was misleading in an earlier version of the manuscript and who has brought to our attention the reference \cite{Sanka-Phd}.}

\address{School of Operations Research and Information Engineering, Cornell University, Ithaca, NY 14853.}
\email{	pp396@orie.cornell.edu}

\author{Y. Zhao}
\address{School of Operations Research and Information Engineering, Cornell University, Ithaca, NY 14853.}
\email{	yz645@cornell.edu}
\usepackage{hyperref}

\title{Spectral decomposition of fractional operators and a reflected stable semigroup}
\begin{document}

	\begin{abstract}
		In this paper, we provide the spectral decomposition  in Hilbert space of the $\mathcal{C}_0$-semigroup $P$ and its adjoint $\hatP$ having as  generator,  respectively, the Caputo and the right-sided Riemann-Liouville fractional derivatives of index $1<\alpha<2$. These linear operators, which are non-local and non-self-adjoint, appear in many recent studies in applied mathematics and also arise as the infinitesimal generators of some substantial processes such as the reflected spectrally negative $\alpha$-stable process.   Our approach relies on intertwining relations that we establish  between these  semigroups and  the semigroup of a  Bessel type process whose generator is a self-adjoint second order differential operator. In particular, from this commutation relation, we characterize  the positive real axis as the continuous point spectrum of $P$ and provide a power series representation of the corresponding eigenfunctions. We also  identify the positive real axis as the residual spectrum of the adjoint operator $\hatP$ and elucidates its role in the spectral decomposition of these operators. By resorting to the concept of continuous frames, we proceed by investigating the domain of the spectral operators and derive two representations for the heat kernels of these  semigroups. As a by-product, we also obtain regularity properties for these latter and also for the solution of the associated Cauchy problem.
	\end{abstract}
\maketitle
\section{Introduction}
Fractional calculus, in which derivatives and integrals of fractional order are defined and studied, is nearly as old as the classical calculus of integer orders. Ever since the first inquisition by L'Hopital and Leibniz in 1695, there has been an enormous amount of study on this topic for more than three centuries, with many mathematicians having suggested their own definitions that fit the concept of a non-integer order derivative. Among the most famous of these definitions are the Riemann-Liouville fractional derivative and the Caputo derivative, the latter being a reformulation of the former in order to use integer order initial conditions to solve fractional order differential equations. In this context, it is natural to consider the following Cauchy problem, for a smooth function $f$ on $x>0$,
\begin{equation} \label{eq:Cauchy_problem}
\begin{cases}
\frac{d}{dt}u(t,x) = \mathbf{D}_{\alpha} u(t,x)\\
u(0,x)=f(x),\\
\end{cases}
\end{equation}
where, for any $1<\alpha<2$, the linear operator $\mathbf{D}_{\alpha}$ is either the Caputo $\alpha$-fractional derivative
\begin{equation} \label{eq:Caputo}
\mathbf{D}_{\alpha}f(x) = {}^CD^{\alpha}_{+}f(x)=\int_0^x\frac{f^{([\alpha]+1)}(y)}{(x-y)^{\alpha-[\alpha]}}\frac{dy}{\Gamma([\alpha]+1-\alpha)},
\end{equation}
with, for  any $k=1,2,\ldots$,   $f^{(k)}(x)=\frac{d^{k}}{dx^k}f(x)$ stands for the $k$-th derivative of $f$, or, the right-sided Riemann-Liouville (RL) derivative
\begin{equation}
\mathbf{D}_{\alpha}f(x) =D^{\alpha}_- f(x)=\left(\frac{d}{dx}\right)^{[\alpha]+1}\int_x^{\infty}\frac{f(y)(y-x)^{[\alpha]-\alpha}}{\Gamma([\alpha]+1-\alpha)}dy,
\end{equation}
with $[\alpha]$ representing the integral part of $\alpha$. We point out that when $\alpha= 2$, in both cases, $\mathbf{D}_2f(x)=\frac12 f^{(2)}(x)$ is a second order differential operator.

In this paper, we aim at providing the spectral representation in $\Lp$ Hilbert space and regularities properties of the solution to the Cauchy problem \eqref{eq:Cauchy_problem}.

The motivation underlying this study are several folds. On the one hand, the last three decades have witnessed the most intriguing leaps in engineering and scientific applications of such fractional operators, including but not limited to population dynamics, chemical technology, biotechnology and control of dynamical systems, and, we refer to the monographs of Kilbas et al.~\cite{Kilbas_fractional_DE}, Meerschaert and Sikorskii \cite{Meerschaert_Book} and  Sankaranarayanan \cite{Sanka-Phd}  for excellent and recent accounts on fractional operators. 
On the other hand, some recent interesting studies have revealed that the linear operator ${}^CD^{\alpha}_{+}$ 
is the infinitesimal generator of $P=(P_t)_{t\geq 0}$ 
the Feller semigroup corresponding to the so-called spectrally negative 
reflected $\alpha$-stable process,  see e.g.~\cite{Meeschart_15,BDP_supremum_stable,Patie_Simon_frac_deriv}. We will provide the formal definition of this process and semigroup in Section~\ref{sec:semigroup}, and, we simply point out that the reflected Brownian motion is obtained in  the limiting case $\alpha= 2$. The reflected $\alpha$-stable processes have been studied intensively in the stochastic processes literature. In particular, we mention that, in a recent paper, Baeumer et.~al.~\cite{Meeschart_15} showed the interesting fact  that the transition kernel of $P$  allows to  map the set of solutions of a Cauchy problem to its fractional (in time) analogue. Motivated by these findings, they provide a numerical method to approximate this transition kernel. In this perspective, in Theorem \ref{theorem:kernel} below, we provide two analytical and simple expressions for this transition kernel.

Although the Cauchy problem for the fractional operators associated to reflected stable processes plays a central role in many fields of sciences, to the best of our knowledge,  their spectral representation remain unclear. This seems to be attributed to the fact that there is not a unified theory for dealing with the spectral decomposition of non-local and non-self-adjoint operators, two properties satisfied, as we shall see in Proposition \ref{prop}, by the fractional operators considered therein. For a nice  account on classical and recent developments on this important topic, we refer to the two volume treatise of  Dunford and Schwartz \cite{Dunford_1963,Dunford1971}  and the monograph  of Davies \cite{Davies_2007}, and the survey paper by Sj\"ostrand \cite{Sjostrand-Survey}.

The purpose of this paper is to provide detailed information regarding the solution of the Cauchy problem \eqref{eq:Cauchy_problem} along with its elementary solution which corresponds to the transition probabilities of the Feller semigroups $P$ and its dual $\hatP$. More specifically, we   provide a spectral representation of this solution in an integral form involving the absolutely continuous part of the spectral measure, the generalized Mittag-Leffler functions as eigenfunctions and a weak Fourier kernel, a terminology which is defined in \cite{Patie_Savov_self_similar} and recalled in  Section~\ref{sec:residual_fucn}.  This kernel admits on a dense subset an integral representation which is given in terms of a function, having a simple expression, that we name a residual function for the dual semigroup (or co-residual function for $P$), as it is associated to elements in its residual spectrum. We refer to Section~\ref{sec:residual_fucn} for more precise definitions. As by-product of this spectral representation, we manage to derive regularity properties for the solution of \eqref{eq:Cauchy_problem} and also for the transition kernel. We already  mention that  we observe a cut-off phenomenon in the nature of the spectrum for the class of operators indexed by the parameter $\alpha \in (1,2]$. Indeed, while the class of Bessel operators which include the limit case  $\mathbf{D}_2$, i.e.~$\alpha=2$,  has  the positive axis $(0,\infty)$ as continuous spectrum, we shall show that this axis corresponds, when $\alpha \in (1,2)$, to the continuous point spectrum of the Caputo operator and the residual spectrum of the right-sided RL fractional operator. 

Our approach relies on an in-depth analysis of an intertwining relation that we establish between the Caputo fractional operator and a second order differential operator of Bessel type, which the latter turns out to be the generator of a  self-adjoint semigroup in $\Lp$. This is combined with the theory of continuous  frames that have been introduced recently in the mathematical physics literature, see \cite{Antoine_frames}.  This work complements nicely the recent works of Patie and Savov in \cite{Patie2015} and \cite{Patie-Savov-15} where such ideas are elaborated between linear operators having a common discrete point spectrum. We also mention that recently Kuznetsov and Kwasnicki \cite{Kuznetsov-K15} provide  a representation of the transition kernel of $\alpha$-stable processes killed upon entering  the negative real line, by inverting their resolvent density that they manage to compute explicitly. In this vein but in a more general context, Patie and Savov in the work in progress \cite{Patie_Savov_self_similar} explore further the idea developed in our paper to establish the spectral theory of the class of positive self-similar semigroups.

The rest of this paper is organized as follows. In Section~\ref{sec:semigroup}, we  introduce the reflected one-sided $\alpha$-stable processes and establish substantial analytical properties of the corresponding semigroups. In Section~\ref{sec:intertwin}, we shall derive the intertwining relation between the spectrally negative reflected stable semigroup and the Bessel-type semigroup. From this link, we extract a set of eigenfunctions that are described  in Section~\ref{sec:eigenfucn} which also includes some of their interesting properties such as the continuous upper frame property, completeness and  large asymptotic behavior. In Section~\ref{sec:residual_fucn}  we investigate the so-called co-residual functions. Finally, in Section~\ref{sec:spectral} we gather all previous results to provide the  spectral  decomposition of the two semigroups $P$ and $\hatP$ including two representations for their transition kernels. The regularity properties are also stated and proved in that Section.
\subsection{Notations}
Throughout, we denote by $\R_+=(0,\infty)$ the positive half-line. For any $-\infty\leq \underline{a}<\overline{a}\leq \infty$, we denote the strip $\C_{(\underline{a},\overline{a})}=\{z\in\C;\:\underline{a}<\Re(z)<\overline{a}\}$, and write simply $\C_+=\C_{[0,\infty)}$. We write $\C_{(-\infty,0)^c}=\{z\in\C;\: \arg(z)\neq \pi\}$  for the  complex plane cut along the negative real axis. We also write $\Lp$ for the Hilbert space of square integrable Lebesgue measurable functions on $\R_+$ endowed with the inner product $\left\langle f,g\right\rangle = \int_0^{\infty}f(x)g(x)dx$ and the associated norm $\|\cdot\|$. For any weight function $\nu$ defined on $\R_+$, i.e.~a non-negative Lebesgue measurable function, we denote by $\Lnu$ the weighted Hilbert space endowed with the inner product $\left\langle f,g\right\rangle_{\nu} = \int_0^{\infty}f(x)g(x)\nu(x)dx$ and its corresponding norm $\|\cdot\|_{\nu}$.
We use $\Cd_0(\R_+)$ to denote the space of continuous real-valued functions on $\R_+$ tending to $0$ at infinity, which becomes a Banach space when endowed with the uniform topology $\|\cdot\|_{\infty}$. Additionally, we denote $\Ctwo$ to be the space of twice continuously differentiable functions on $\R_+$, which vanishes at both $0$ and infinity, and $\Cd^{\infty}(\R_+)$ the space of functions with continuous derivatives on $\R_+$ of all orders, and ${\mathrm{B}}_b(\R_+)$ the real-valued bounded Borel measurable functions on $\R_+$.
For Banach spaces $\rmH_1,\rmH_2$, we define
\[ \mathbf{B}(\rmH_1,\rmH_2)=\{L:\rmH_1\rightarrow \rmH_2 \textnormal{ linear and continuous mapping}\}.\]
In the case of one Banach space $\rmH$, the unital Banach algebra $\mathbf{B}(\rmH,\rmH)$ is simply denoted by $\mathbf{B}(\rmH)$. Moreover, a semigroup $P=(P_t)_{t\geq 0}$ where $P_t\in\mathbf{B}(\rmH)$ is called a positive $\mathcal{C}_0$-semigroup on $\rmH$ if $P_{t+s}=P_t\circ P_s$, $P_t f\geq 0$ for $f\geq 0$, and for any functions $f\in \rmH$, $\|P_tf-f\|_{\rmH}\rightarrow 0$ as $t\rightarrow 0$. In the case when $\rmH=\Cd_0(\R_+)$ endowed with the uniform topology, we say $P$ is a Feller semigroup on $\R_+$. Furthermore, for an operator $T\in\mathbf{B}(\rmH_1,\rmH_2)$, we use the notation $\Ran(T)$ (resp.~$\Ker(T)$) for the range (resp.~the kernel) of $T$ and $\overline{\Ran}(T)$ (resp.~$\overline{\Ker}(T)$) for its closure. For any set of functions ${\mathrm{E}}\subseteq \rmH$, we use ${\mathrm{Span}}({\mathrm{E}})$ to denote the set of all linear combinations of functions in ${\mathrm{E}}$, and $\overline{{\mathrm{Span}}}({\mathrm{E}})$ for its closure.
We now proceed to define a few further notations. For two functions $f,g:\R_+\rightarrow \R$, we write $f\stackrel{a}{=}\bo(g)$  (resp.~$f\stackrel{a}{=}\so(g)$) if $\limsup_{x\to a}\frac{f(x)}{g(x)}<\infty$ (resp.~$\lim_{x\to a}\frac{f(x)}{g(x)}=0$), and
$ f\asymp g$  (resp.~$f \,{\buildrel a \over \sim}\, g$) if $\exists \: c>0 \textnormal{ such that } c\leq \frac{f(x)}{g(x)}\leq c^{-1} \textnormal{ for all } x\in\R_+$ (resp.~if $\lim_{x\rightarrow a}\frac{f(x)}{g(x)}=1$  for some $a\in \R\cup \{\pm \infty\}$).
Finally, for any $q\in\R_+$, we write $\rmdq f(x)=f(qx)$ and for any $\alpha,\tau>0$, we set
\begin{equation}\label{eq:def_e}
\e_{\alpha,\tau}(x)={\rm{d}}_{\tau^{\frac1\alpha}}\e_{\alpha}\left(x\right)=e^{-\tau x^{\alpha}},\quad x>0.\end{equation}

\section{Fractional operators and the reflected stable semigroup} \label{sec:semigroup}
Let $Z= (Z_t)_{t\geq 0}$ be a spectrally negative $\alpha$-stable L\'evy process with $\alpha\in (1,2)$, defined on a filtered probability space $(\Omega,\Filt,(\Filt_t)_{t\geq 0},\Prob=(\Prob_x)_{x\in\R})$.  It means that $Z$ is a process with stationary and independent increments, having no positive jumps, and its law is characterized, for $t>0$, by
\begin{equation} \label{eq:lapl_Z}
\log \E[e^{zZ_t}] = z^{\alpha}t, \quad z\in \C_+. \end{equation}
Here and below $z^{\alpha}$ is the main branch of the complex analytic function in the complex half-plane $\Re(z)\geq0$, so that $1^{\alpha} = 1$. Let $X=(X_t)_{t\geq 0}$ be the process $Z$ reflected at its infimum, that is, for any $t\geq0$,
 \[ X_t = 	\left\{\begin{array}{lcl}Z_t & \mbox{if} & t< T^Z_{(-\infty,0]}, \\
 Z_t - \inf_{s\leq t}Z_s & \mbox{if} & t\geq T^Z_{(-\infty,0]}, \end{array}\right.\]
with $T^Z_{(-\infty,0]} = \inf\{t>0;\: Z_t\leq 0\}$, and we write, for any $f\in \Bb,t,x\geq 0$,
 \begin{eqnarray}
 P_t f(x) &=&\E_x\left[f(X_t)\right],
 \end{eqnarray}
where $\E_x$ stands for the expectation operator associated to $\Prob_x(Z_0=x)=1$.
Next, let $\hatZ=-Z$ be the dual process of $Z$ (with respect to the Lebesgue measure), which is a spectrally positive $\alpha$-stable process,   and, let $\hatX=(\hatX_t)_{t\geq 0}$ be the process defined from $ \hatZ$ by a random time-change  as follows,  for any $t\geq0$, 
\begin{equation}\label{eq:def_hatx}
  \hatX_t= \hatZ_{\hattau_t},
\end{equation}
where $\hattau_t=\inf\{u>0;\:\widehat{A}_u>t\}$ and $\widehat{A}_t=\int_0^{t}\mathbb{I}_{\{\hatZ_s>0\}}ds$.
We also write for any $f\in \Bb,t,x\geq 0$,
\[\hatP_t f(x) =\widehat{\E}_x[f(\hatX_t)],\]
where $\widehat{\E}_x$ stands for the expectation operator associated to $\widehat{\Prob}_x(\hatZ_0=x)=1$. We are now ready to state our first result.
\begin{proposition} \label{prop}
\begin{enumerate}
\item $P$ is a positive contractive $\mathcal{C}_0$-semigroup on $\Co$, i.e.~a Feller semigroup, whose infinitesimal generator is $({}^CD^{\alpha}_{+},\mathcal{D}_{\alpha})$ where
\[ \mathcal{D}_{\alpha}=\left\{f \in\Co; \:f(x)=\int_0^{\infty} \left(e^{-y} \calJ(x)-\calJ^{'}(x-y)\mathbb{I}_{\{y<x\}}\right)g(y)dy ,  g\in \Co \right\}, \] with
\begin{eqnarray}
\calJ(z)&=&\frac{1}{\Gamma(1+\frac1\alpha)}\sum_{n=0}^{\infty}\frac{(e^{i\pi}z^{\alpha})^n}{\Gamma(\alpha n+1)},\quad z\in\mathbb{C},\label{eq:calJr}
\end{eqnarray}
which is easily seen to define a  function holomorphic on $\C_{(-\infty,0)^c}$.
\item $P$ admits a unique extension as a contractive $\mathcal{C}_0$-semigroup on $\Lp$, which is also denoted by  $P=(P_t)_{t\geq0}$ when there is no confusion (otherwise we may denote $P^F$ for the Feller semigroup). The domain of its infinitesimal generator $L^X$ is given by
\begin{equation}\label{eq:domain_LX}
\D_{\alpha}(\Lp)=\left\{f\in \Lp; \: \int_{-\infty}^{\infty}\left|\Fou^{+}_f(\xi)\right|^2|\xi|^{2\alpha}d\xi<\infty \right\} 
\end{equation}
where  $\Fou^{+}_f(\xi)=\int_{0}^{\infty}e^{i\xi x}f(x)dx$ is the one-sided Fourier transform of $f$ taken  in the $\rm{L}^2$ sense. 

\item  \label{it:dual} $\hatX$ is the (weak) dual of $X$ with respect to the Lebesgue measure. Moreover, $\hatP$ is a Feller semigroup which 
 admits a unique extension as a contractive $\mathcal{C}_0$-semigroup on $\Lp$, also denoted by $\hatP$, which has $(D^{\alpha}_{-}, \D_{\alpha}(\Lp))$  as infinitesimal generator. 
Clearly as $P \neq \hatP$, we get that $P$ is non-self-adjoint in $\Lp$.
\end{enumerate}
\end{proposition}
\begin{remark}
We point out that when $\alpha=2$, $P$ is the 1-dimensional Bessel semigroup, see \cite[Appendix 1]{Borodin-Salminen-02}, which also belongs to the class of the so-called $\alpha$-Bessel semigroups, which are reviewed in more details in Appendix~\ref{sec:appen}. In this case, $\hatP=P$ and $P$ is self-adjoint in $\Lp$.
\end{remark}
\begin{remark}
Note that the function $\calJ(e^{i\pi}z^{\frac1\alpha})$ is the (generalized) Mittag-Leffler function of parameters $(\alpha,1)$, see e.g.~\cite{Kilbas_fractional_DE} for a detailed account on this function.
\end{remark}
In order to prove this Proposition, we first state and prove the following lemma, which generalizes \cite[Lemma 2]{Bertoin-92} and may be of independent interests.
\begin{lemma} \label{lemma:time_change}
Let $Y_t=Z_{\tau_t}, t\geq 0$, where $\tau_t=\inf\{u>0;\:A_u>t\}$ and $A_t=\int_0^{t}\mathbb{I}_{\{Z_s>0\}}ds$. Then  $(Y_t)_{t\geq 0}$ is a $(\Filt_{\tau_t})_{t\geq 0}$ strong Markov process and for any $f\in\Bb,t,x\geq 0$, we have
\begin{equation} \label{eq:Y_tau_X}
P_t f(x)=\E_x[f(Y_t)].
\end{equation}
Moreover, $(Y_t)_{t\geq 0}$ and $(\hatX_t)_{t\geq 0}$ are dual processes with respect to the Lebesgue measure.
\end{lemma}
\begin{proof}
For any $f\in\Bb,q>0$, let
\[U_qf(x)=\int_0^{\infty}e^{-qt}P_tf(x)dt,\quad U^{\dag}_qf(x)=\int_0^{\infty}e^{-qt}\E_x\left[f(X_t)\mathbb{I}_{\{t<T^X_0\}}\right]dt\]
be the resolvents of $X$ and  $X^{\dag}=(X^{\dag}_t)_{t\geq 0}$, the process $X$ killed at time $T_0^X=\inf\{t>0;\: X_t = 0\}$, respectively. It is easy to observe from the construction of $X$ that $T_0^X=T^Z_{(-\infty,0]}$. Moreover, by \cite[Example 3]{Rivero05}, $X$ can also be defined as the unique self-similar recurrent extension of $X^{\dag}$ and we get, from an application of the strong Markov property, that for all $x\geq 0$,
\begin{equation}
U_qf(x)=U^{\dag}_qf(x)+\E_x\left[e^{-q T^X_0}\right]U_qf(0).
\end{equation}
Next, since $Z$ has  paths of unbounded variation, by \cite[Theorem 6.5]{Kyprianou-06}, we have $\Prob_x(T^Z_{[0,\infty)}=0) =1$ for $x\geq0$ and $\Prob_x(T_{[0,\infty)}^Z>0)=1$ for any $x<0$, where $T_{[0,\infty)}^Z=\inf\{t>0;\: Z_t\geq 0\}$. Thus, the fine support of the additive functional $(A_t)_{t\geq 0}$, defined as the set $\{x\in\R;\: \Prob_x(\tau_0=0)=1\}$, is plainly $[0,\infty)$. Moreover, as the L\'evy process $Z$ is a Feller process and therefore a Hunt process (see e.g.~\cite[Section 3.1]{Chung_Walsh_Markov}), we have from \cite{Gzyl1980} that $(Y_t)_{t\geq 0}$ is a $(\Filt_{\tau_t})_{t\geq 0}$ strong Markov process, whose resolvent is defined, for $f\in\Bb$, by
\[V_qf(x)=\int_{0}^{\infty}e^{-qt}\E_x[f(Y_t)]dt.\]
Furthermore, it is easy to observe that $A_t=t$ for any $t\leq T^Z_{(-\infty,0]}$ and thus $\tau_t=t$  for any $t< T^Z_{(-\infty,0]}$. On the other hand, since $Z$ is a spectrally negative L\'evy process with no Gaussian component, $Z$ does not creep below, see e.g.~\cite[Exercise 7.4]{Kyprianou-06}, and therefore $T_0^Z =\inf\{t>0;\: Z_t=0\} > T^Z_{(-\infty,0]} \: a.s.$, where $a.s.$ throughout this proof, means $\Prob_x$-almost surely for all $x>0$. Moreover, observe that a.s.
\[A_{T^Z_0}=\int_0^{T^Z_{(-\infty,0]}} \mathbb{I}_{\{Z_s>0\}}ds + \int_{T^Z_{(-\infty,0]}}^{T^Z_0} \mathbb{I}_{\{Z_s>0\}}ds = A_{T^Z_{(-\infty,0]}}=T^Z_{(-\infty,0]} .\]
Next, recalling that $T^Z_{(-\infty,0]}=T^X_0$, we deduce from the previous identity that, with the obvious notation, a.s.
\begin{equation}\label{eq:T0X_T0Y}
T_0^Y=A_{T^Z_0}=T^Z_{(-\infty,0]}=T^X_0.
\end{equation}
Since it is clear that $Y_t=Z_{\tau_t}=Z_t=X_t$ for $t<T^X_0$, we have for any $f\in\Bb$ and $q>0$,
\[V^{\dag}_qf(x)=\int_0^{\infty}e^{-qt}\E_x\left[f(Y_t)\mathbb{I}_{\{t<T^Y_0\}}\right]dt=\int_0^{\infty}e^{-qt}\E_x\left[f(X_t)\mathbb{I}_{\{t<T^X_0\}}\right]dt=U^{\dag}_qf(x).\]
Hence, the strong Markov property of $(Y_t)_{t\geq 0}$ together with \eqref{eq:T0X_T0Y} yield that, for every $x\geq 0$,
\[V_qf(x)=V^{\dag}_qf(x)+\E_x\left[e^{-qT^Y_0}\right]V_qf(0)=U^{\dag}_qf(x)+\E_x\left[e^{-qT^X_0}\right]V_qf(0).\]
Next, according to \cite[Lemma 2]{Bertoin-92} and after an obvious dual argument, $(Y_t)_{t\geq 0}$ and $(X_t)_{t\geq 0}$ have the same law under $\Prob_0$ and therefore $V_qf(0)=U_qf(0)$. Hence
\[U_qf(x)=U^{\dag}_qf(x)+\E_x\left[e^{-qT^X_0}\right]U_qf(0)=U^{\dag}_qf(x)+\E_x\left[e^{-qT^X_0}\right]V_qf(0)=V_qf(x),\]
which proves the identity \eqref{eq:Y_tau_X}. Next, by \cite[Proposition 4.4]{Walsh_1972}, we observe that $(A_t)_{t\geq 0}$ and $(\widehat{A}_t)_{t\geq 0}$ are dual additive functionals, both of which are finite for each $t$ and continuous. Hence by \cite[Theorem 4.5]{Walsh_1972}, $(Y_t)_{t\geq 0}$ and $(\hatX_t)_{t\geq 0}$ are dual processes with respect to the Revuz measure associated to $A$, which, by \cite{Revuz_1970}, is the Lebesgue measure. This completes the proof of this lemma.
\end{proof}
\begin{proof}[Proof of Proposition~\ref{prop}]
The Feller property of the semigroup $P$ is given in \cite[Proposition VI.1]{BertoinLevy}. Moreover, the fact that the infinitesimal generator of $P$ is ${}^CD^{\alpha}_{+}$ has been proved in various papers, see e.g.~\cite{BDP_supremum_stable} and \cite{Patie_Simon_frac_deriv}, and the domain $\D_{\alpha}$ is given in \cite[Proposition 2.2]{Patie_Simon_frac_deriv}, which completes the proof of the first item. Next, from \cite[Lemma 3]{Rivero05} and its proof, we know that, up to a multiplicative positive constant, the Lebesgue measure is the unique excessive measure for $P$, where with the notation of \cite[Example 3]{Rivero05}, $\gamma=1-\frac1\alpha$. Thus, since $X$ is stochastically continuous, see \cite[Lemma 2.1]{Lamperti-72}, a classical result from the general theory of Markov semigroups, see e.g.~\cite[Theorem 5.8]{Daprato-06}, yields that the Feller semigroup $P$ admits a unique extension as a contractive $\mathcal{C}_0$-semigroup on $\Lp$. We now proceed to characterize the domain of the infinitesimal generator of the $\Lp$-extension, denoted by $\D^X$. To this end, we first observe from \cite[Theorem 12.16]{Berg-Forst} that since $Z$ is a L\'evy process, its semigroup $(\overline{P}_t)_{t\geq 0}$, i.e.~$\overline{P}_tf(x)=\E_x[f(Z_t)],x\in\R$, is a $\Lr$-Markov semigroup, and 
 its infinitesimal generator, denoted by $L^Z$, has the following anisotropic Sobolev space as its domain
\begin{equation}\label{eq:domain_LZ}
\D^Z = \left\{\bar{f}\in \Lr; \: \int_{-\infty}^{\infty}\left|\Fou_{\bar{f}}(\xi)\right|^2|\xi|^{2\alpha} d\xi<\infty\right\},
\end{equation}
where  $\Fou_{\bar{f}}(\xi)=\int_{-\infty}^{\infty}e^{i\xi x}\bar{f}(x)dx$ is the Fourier transform of $\bar{f}$. Now for a function $f$ on $\R^+$ we define its extension $\bar{f}:\R\rightarrow \R$   as $\bar{f}(x)=f(x)\mathbb{I}_{\{x> 0\}}$. Then,  for   any $f\in\overline{\D}_{\alpha}(\Lp) =\{ f \in \D_{\alpha}(\Lp);\: \bar{f}\in \Ctwor\}$, we have clearly $\bar{f}\in\D^Z \cap \Ctwor$ and thus by combining \cite[Section I.2]{BertoinLevy} and \cite[Theorem 2.1]{Gzyl1980} we get, that for any $x>0$,
\begin{equation}
L^Xf(x)=a(x)L^Z\bar{f}(x),
\end{equation}
where $a(x)=\mathbb{I}_{\{x> 0\}}$ from \cite[(3.6)]{Gzyl1980}. Therefore, since $L^Z\bar{f}\in\Lr$, it is obvious that $L^Xf\in \Lp$, which implies that $f\in \D^X$. Next, for any $\tau>0$, let $f_{\tau}(x)=\tau^3 x^3e^{-\tau x}, x>0$, then easy computations yield that  for all $\tau>0$ $f_{\tau}\in \overline{\D}_{\alpha}(\Lp)$, hence by the Wiener's theorem for Mellin transform $\overline{\D}_{\alpha}(\Lp)$ is dense in $\Lp$ and therefore, for any $f\in \D_{\alpha}(\Lp)$, we can take $(f_n)_{n\geq 0} \subset \overline{\D}_{\alpha}(\Lp) \cap \Ctwo$ such that $f_n\rightarrow f$ in $\Lp$. Writing $\bar{f}_n$ and $\bar{f}$ their corresponding extensions to $\Lr$ as above, we still have $\bar{f}_n\rightarrow \bar{f}$ in $\Lr$ and  $\bar{f}\in D^Z$. Also note that for each $\xi\in\R$,
\begin{equation}\label{eq:Fourier_gen}
\Fou^{+}_{L^X f_n}(\xi)=\Fou_{L^Z \bar{f}_n}(\xi)=(-i\xi)^{\alpha}\Fou_{\bar{f}_n}(\xi)\rightarrow (-i\xi)^{\alpha}\Fou_{\bar{f}}(\xi)=\Fou_{L^Z \bar{f}}(\xi),
\end{equation}
where we used \cite[Theorem 12.16]{Berg-Forst} for the second and last identity. Therefore $L^X f_n$ converges in $\Lp$ and $f\in \D^X$ by the closedness of infinitesimal generator. This shows that $\D_{\alpha}(\Lp)\subseteq\D^X$. On the other hand, take now $f\in\D^X \cap \Ctwo$ and let $\bar{f}$ be constructed as above. Then by \cite[Theorem 2.6]{Gzyl1980} and recalling that the fine support of $(A_t)_{t\geq 0}$ is $\R_+$, we have
\begin{equation} \label{eq:LZ_LX}
L^Z\bar{f}(x) = 	\left\{\begin{array}{lcl}b(x)L^Xf(x) & \mbox{for} & x\geq 0, \\
0 & \mbox{for} & x<0, \end{array}\right.
\end{equation}
where, denoting $\mathbb{I}_{+}(x)=\mathbb{I}_{\{x>0\}}$, \[b(x)=\lim_{t\rightarrow 0}\frac{\E_x[\int_0^t\mathbb{I}_{\{Z_s>0\}}ds]}{t}=\lim_{t\rightarrow 0}\frac{\int_0^t \overline{P}_s \mathbb{I}_{+}(x)ds}{t}=\lim_{t\rightarrow 0}\overline{P}_t\mathbb{I}_{+}(x)=\mathbb{I}_{\{x> 0\}}\] for each $x\in\R$. Therefore, we have
\begin{align*}
\int_\R\left(L^Z\bar{f}(x)\right)^2dx&=\int_\R\left((\mathbb{I}_{\{x< 0\}}+\mathbb{I}_{\{x\geq 0\}})L^Z\bar{f}(x)\right)^2dx
= \int_0^{\infty} \left(L^Z\bar{f}(x)\right)^2dx=\int_0^{\infty} \left(L^X f(x)\right)^2dx, 
\end{align*}
which implies that $\bar{f}\in \D^Z$ and $f\in\D_{\alpha}(\Lp)$. Next, since we have proved that $\D_{\alpha}(\Lp)\subseteq \D^X$ and $\D_{\alpha}(\Lp) \cap \Ctwo$ is dense in $\Lp$, we have that $ \D^X \cap \Ctwo$ is also dense in $\Lp$. Hence the same argument as above shows that \eqref{eq:LZ_LX} still holds for any $f\in \D^X$, which further proves that $\D^X\subseteq \D_{\alpha}(\Lp)$ and completes the proof for the second argument. For the duality argument, we first observe  from Lemma~\ref{lemma:time_change} that $X$ and $\hatX$ are dual processes with respect to the Lebesgue measure. Moreover, note that the minimal process $X^{\dag}$ belongs to the class of \textit{ positive $\frac1\alpha$-self-similar Markov processes} as introduced in \cite{Lamperti-72}, which also provides a bijection between positive self-similar processes and L\'evy processes stated as follows. Let us define, for any $t\geq0$, $\vartheta_{t}=\inf\{u>0;\int_0^u (X_s^{\dag})^{-\alpha}ds > t \}$, then the process
\begin{equation}\label{eq:defk}
\xi^{\dag}_t = \log X^{\dag}_{\vartheta_t},\end{equation}
is a L\'evy process killed at an independent exponential time. More specifically, by \cite{Patie_Simon_frac_deriv}, the Laplace exponent of $\xi^{\dag}$ is
\begin{equation} \label{eq:psi0}
\psi^{\dag}(u)=\frac{\Gamma(u+1)}{\Gamma(u-\alpha+1)}, \:u>-1.
\end{equation}
Note that by writing  $\theta$ for the largest non-negative root of the convex function $\psi^{\dag}$, it is easy to check that $\theta=\alpha-1\in(0,1)$. Hence by \cite[Section 5]{Rivero05}, there exists a dual process of $X^{\dag}$, denoted by $\hatX^{\dag}$, with the Lebesgue measure serving as the reference measure. Moreover, $\hatX^{\dag}$ is also a positive $\frac1\alpha$-self-similar process with its corresponding L\'evy process denoted by $\widehat{\xi}^{\dag}$, which is  the dual of the L\'evy process obtained from $\xi^{\dag}$ by means of Doob $h$-transform via the invariant function $h(x)=e^{\theta x}, x\in \R$. Therefore, the Laplace exponent of $\widehat{\xi}^{\dag}$ takes the form, for $u<0$,
\begin{equation*}
\widehat{\psi}(u)=\psi^{\dag}(-u+\theta)=\psi^{\dag}(-u+\alpha-1)=\frac{\Gamma(\alpha-u)}{\Gamma(-u)}.
\end{equation*}
Note that $\widehat{\xi}^{\dag}$ drifts to $-\infty$ a.s. and thus $\hatX^{\dag}$ has a a.s. finite lifetime $T_0^{\hatX^{\dag}}=\inf\{t>0;\hatX^{\dag}_t\leq 0\}$. Hence by recalling that $X$ can be viewed as the  recurrent extension of $X^{\dag}$ that leaves $0$ continuously a.s., we deduce from \cite[Lemma 6]{Rivero05} that $\hatX$ can also be viewed as the recurrent extension of $\hatX^{\dag}$ which leaves $0$ by a jump according to the  jump-in measure $C x^{-\alpha}, x,C>0$. The Feller property of the semigroup of such recurrent extension has been shown in \cite[Proposition 3.1]{Blumenthal_83}, while the existence of the $\Lp$-extension follows by the same argument than the one we developed  for $P$. 
Moreover, from \cite[Theorem 12.16]{Berg-Forst}, we deduce easily that $\D^{\hatZ}=\D^Z$, hence using the same method as above, we get that $\D^{\hatX}=\D^X=\D_{\alpha}(\Lp)$. Finally, using the same arguments as in \eqref{eq:Fourier_gen}, we see that for any $f\in \D_{\alpha}(\Lp)$,
\[\Fou^{+}_{L^{\hatX} f}(\xi)=\Fou_{L^{\hatZ} \bar{f}}(\xi)=(i\xi)^{\alpha}\Fou_{\bar{f}}(\xi).\]
Comparing this identity with \cite[Lemma 2.1 and Theorem 2.3]{Ervin}, we conclude that $L^{\hatX} f=D_{-}^{\alpha} f$ on $ \D_{\alpha}(\Lp)$. This completes the proof.
\end{proof}

\section{Intertwining relationship} \label{sec:intertwin}
We say that a linear operator $\Lambda$ is a \textit{multiplicative operator} if it admits the following representation, for any $f \in \mathrm{B}_b(\R_+)$,
\[ \Lambda f(x)=\int_0^{\infty}f(xy)\lambda(y)dy, \]
for some integrable function $\lambda$. When in addition $\lambda$ is the density of the law of a random variable ${\bf{X}}$, i.e.~$\lambda(y)\geq 0$ and $\spnu{1,\lambda}{}=1$, we say that $\Lambda$ is a \textit{Markov multiplicative operator}.
Moreover, $\M_{\lambda}= \M_{\Lambda}=\M_{\mathbf{X}}$ is called a Markov multiplier where for at least $\Re(s)=1$,
\[ \M_{{\Lambda}}(s)=\int_0^{\infty}y^{s-1}\lambda(y)dy, \]
is the Mellin transform of  $\lambda$. By adapting the developments in \cite[2.1.9]{Titchmarsh_86} based on the Fourier transform,  we also have that if $\int_0^{\infty}y^{-\frac12}\lambda(y)dy <\infty$ then  $\Lambda \in \B{\Lp}$ with, for any $f \in \Lp$,
\begin{equation}\label{eq:mm}
 \M_{{\Lambda f}}(s)=\M_{{\Lambda}}(1-s)\M_{f}(s).
 \end{equation}
 Note that this latter provides that $\Lambda$ is one-to-one in $\Lp$ if  $\M_{{\Lambda}}(1-s)\neq 0$.
We also recall from \cite{Paris01} that if $s\mapsto \M_{\lambda}(s)$ is defined, absolutely integrable and uniformly decays to zero along the lines of the strip $s\in\C_{(\underline{a},\overline{a})}$ for some $\underline{a}<\overline{a}$, then the Mellin inversion theorem applies to yield, for any $x>0$,
\begin{equation} \label{eq:Mellin_inversion}
\lambda(x)=\frac{1}{2\pi i}\int_{a-i\infty}^{a+i\infty}x^{-s}\M_{\lambda}(s)ds, \quad \underline{a}<a<\overline{a}.
\end{equation}
 Now we are ready to state the following.
\begin{theorem} \label{thm:intertwin}
Let us write, for any $\alpha \in (1,2)$,
\begin{equation} \label{eq:defn_La}
\M_{\La}(s)=\frac{\Gamma(\frac{s-1}{\alpha}+1)\Gamma(\frac{s}{\alpha})}{\Gamma(\frac1\alpha)\Gamma(s)}, \quad s \in \C_+.
\end{equation}
Then, the following holds.
\begin{enumerate}
\item \label{it:Ma} $\M_{\La}$ is a Markov multiplier and  $\La\in \mathbf{B}(\Lp)\cap \mathbf{B}(\Co)$.  Moreover, it is one-to-one  on $\Co$, and,  in $\Lp$, $\overline{\Ran}(\La)=\Lp$.
\item \label{it:int} Moreover, for any $t\geq 0$ and $f\in \Lp$, the following intertwining relation holds
\begin{equation} \label{eq:intertwin}
P_t\La f=\La Q_t f,
\end{equation}
where $Q=(Q_t)_{t\geq 0}$ is the $\Lp$-extension of the $\alpha$-Bessel self-adjoint semigroup as defined in Appendix~\ref{sec:appen}.
\item \label{it:intg} Consequently, we have, for any $f\in\D_{\mathbf{L}}(\Lp)$,
\begin{equation}
{}^CD^{\alpha}_+ \La f = \La \mathbf{L} f,
\end{equation}
where the fractional operator  ${}^CD^{\alpha}_+$  was defined in \eqref{eq:Caputo}, while the second order differential operator $\mathbf{L}$ and its $\Lp$-domain $\D_{\mathbf{L}}(\Lp)$ are defined in \eqref{eq:L_Q} and \eqref{eq:L2domain_bfL}, respectively.
\end{enumerate}
\end{theorem}
The proof of this Theorem is split into three steps. First, we show that \eqref{eq:defn_La} is indeed a Markov multiplier. Then, we establish the identity \eqref{eq:intertwin} in the space $\Co$. Finally, by remarking that  $\Co$ is dense in $\Lp$, we can extend the intertwining identity to $\Lp$ by a continuity argument.

\subsection{The Markov multiplicative operator $\La$} \label{section:Markov_kernel}
In order to  prove Theorem~\ref{thm:intertwin}\eqref{it:Ma}, which provides some substantial properties of $\La$, we shall need the following claims.
\begin{lemma}\label{lemma:La_G_e}
Let us  define
\begin{eqnarray}
	g_{\alpha}(z)&=&\sum_{n=0}^{\infty} \frac{\Gamma\left(\frac1\alpha\right) \Gamma(\alpha n+1)}{\Gamma(n+\frac1\alpha)( n!)^2} (e^{i\pi} z^{\alpha})^n, \label{eq:r}
\end{eqnarray}
then $g_{\alpha}$ is holomorphic on $\C_{(-\infty,0)^c}$. Moreover, $g_{\alpha}\in\Lp$ with $\La  g_{\alpha}= \e_{\alpha}$ where $\e_{\alpha}$ is defined in \eqref{eq:def_e}.
\end{lemma}
\begin{proof}
First, from the Stirling approximation
\begin{equation}\label{eq:saa}
\Gamma(a)\siminfy \sqrt{2\pi} a^{a-\frac12}e^{-a},
 \end{equation}
see \cite[(1.4.25)]{Lebedev-72},  we  get that $\frac{\Gamma(\alpha n+\alpha+1)\Gamma(n+\frac1\alpha)}{\Gamma(\alpha n+1)\Gamma(n+1+\frac1\alpha)(n+1)^2} \eqinfy \bo(n^{\alpha-3})$, hence, as $\alpha \in (1,2)$, $g_{\alpha}$ is holomorphic on $\C_{(-\infty,0)^c}$. We now proceed to show that $g_{\alpha}\in \Lp$. To this end, let us define, for $0<\Re(s)<1$,
\begin{equation*}
	\M_{\alpha}\left(s\right)=\frac{\Gamma(1+\frac1\alpha)\Gamma(\frac{s}{\alpha})\Gamma(1-s)}{\Gamma(1-\frac{s}{\alpha})\Gamma(\frac{1-s}{\alpha})}
\end{equation*}
and we first aim at proving that $\M_{\alpha}=\M_{g_{\alpha}}$ the Mellin transform of $g_{\alpha}$.
For this purpose, observe that $s\mapsto \M_{\alpha}(s)$ is holomorphic on $\C_{(0,1)}$ and then consider the contour integral $I_{N,B}=\frac{1}{2\pi i}\int_{C_{N,B}}z^{-s}\M_{{\alpha}}\left(s\right)ds$ where $C_{N,B}$ is the rectangle with vertices at $\frac12\pm iB$ and $-\alpha N-\frac{\alpha}{2}\pm iB$ for some large $N\in\mathbb{N}$ and $B>0$. Then we can obviously split $I_{N,B}$ into four parts, namely $I_{N,B}=I_1+I_2+I_3+I_4$ where
\begin{align*}
I_1=\frac{1}{2\pi i}\int_{\frac12+iB}^{-\alpha N-\frac{\alpha}{2}+iB}z^{-s}\M_{\alpha}(s)ds, &\quad I_2=\frac{1}{2\pi i}\int_{-\alpha N-\frac{\alpha}{2}+iB}^{-\alpha N-\frac{\alpha}{2}-iB}z^{-s}\M_{\alpha}(s)ds,\\
I_3=\frac{1}{2\pi i}\int_{-\alpha N-\frac{\alpha}{2}-iB}^{\frac12-iB}z^{-s}\M_{\alpha}(s)ds, &\quad I_4=\frac{1}{2\pi i}\int_{\frac12-iB}^{\frac12+iB}z^{-s}\M_{\alpha}(s)ds.
\end{align*}
Next, observing  from the  Stirling approximation,  see e.g.~\cite[(2.1.8)]{Paris01}, that for fixed $a\in\R$,
\begin{equation} \label{eq:asympt_gamma}
|\Gamma(a+ib)|\sipminfy C |b|^{a-\frac12}e^{-\frac{\pi}{2}|b|},
\end{equation}
with $C=C(a)>0$, we deduce, for some $C_{\alpha}>0$, that
\begin{equation} \label{eq:asymp_Ma}
\left|\M_{\alpha}\left(a+ib\right)\right| \sipminfy C_{\alpha} |b|^{\frac3\alpha a-a-\frac1\alpha} e^{-\frac{\pi}{2}(1-\frac1\alpha)|b|},
\end{equation}
and, hence
\begin{equation} \label{eq:asymp_Mga}
	\left|z^{-(a+ib)}\M_{{\alpha}}\left(a+ib\right)\right|\sipminfy C_{\alpha} |z|^{-a}|b|^{\frac3\alpha a-a-\frac1\alpha} e^{-\frac{\pi}{2}(1-\frac1\alpha)|b|+\arg(z) b}.
\end{equation}
Therefore, if $|\arg(z)|<\frac{\pi}{2}(1-\frac1\alpha)$ and  $N$ is kept fixed,  we have both
\begin{equation}\label{eq:I0}
\lim_{B\rightarrow\infty} |I_1| =\lim_{B\rightarrow\infty} |I_3|= 0.
\end{equation} For the integral $I_2$, we have
\begin{align*}
 |I_2|&\leq  \frac{1}{2\pi}|z|^{\alpha N+\frac{\alpha}{2}} \int_{-\infty}^{\infty}e^{\arg(z) b}\left|\frac{\Gamma(1+\frac1\alpha)\Gamma(-N-\frac12+i\frac{b}{\alpha})\Gamma(1+\alpha N+\frac{\alpha}{2}+ib)}{\Gamma(N+\frac32-i\frac{b}{\alpha})\Gamma(N+\frac12+\frac1\alpha-i\frac{b}{\alpha})}\right|db \\
 &= \frac{1}{2}|z|^{\alpha N+\frac{\alpha}{2}} \int_{-\infty}^{\infty}e^{\arg(z) b}\left|\frac{\Gamma(1+\frac1\alpha)\Gamma(1+\alpha N+\frac{\alpha}{2}+ib)}{\Gamma(N+\frac32-i\frac{b}{\alpha})^2\Gamma(N+\frac12+\frac1\alpha-i\frac{b}{\alpha})\cosh(\frac{\pi b}{\alpha})}\right|db
\end{align*}
where we have used the reflection formula for the gamma function. Using the Stirling approximation again, it is easy to derive, for large $N$, the  upper bound
\begin{align*}
\left|\frac{\Gamma(1+\alpha N+\frac{\alpha}{2}+ib)}{\Gamma(N+\frac32-i\frac{b}{\alpha})^2}\right|\leq C e^{N(\alpha\log\alpha-\alpha-2)}N^{(\alpha-2)N+\frac{\alpha-1}{2}}
\end{align*}
which is uniform in $b\in\R$ and where $C>0$. Moreover, recalling, from \cite[(5.1.3)]{Paris01},  that, for $N\geq 1$ and $b\in\R$,
$\left|\Gamma(N+\frac12+\frac1\alpha-i\frac{b}{\alpha})\right|\geq \frac{\Gamma(N+\frac12+\frac1\alpha)}{\cosh^{\frac12} (\frac{\pi b}{\alpha})}$,
	we find
\begin{align*}
|I_2|&\leq C \frac{e^{N(\alpha\log\alpha-\alpha-2)}N^{(\alpha-2)N+\frac{\alpha-1}{2}}}{\Gamma(N+\frac12+\frac1\alpha)}\int_{-\infty}^{\infty}\frac{e^{\arg(z) b}}{\cosh^{\frac12} (\frac{\pi b}{\alpha})}db
\end{align*}
where  the last integral converges absolutely whenever $|\arg(z)|<\frac{\pi}{2\alpha}$. For such $z$, since   $1<\alpha<2$, we get that $\lim_{N\rightarrow\infty}|I_2|= 0$. Therefore, combining this with \eqref{eq:I0}, we have, for $|\arg(z)|<\frac{\pi}{2\alpha} \wedge \frac{\pi}{2}(1-\frac1\alpha)=\frac{\pi}{2\alpha} $,
\[\lim_{N,B\rightarrow \infty} I_{N,B}=\lim_{B\rightarrow \infty} I_4 = \frac{1}{2\pi i}\int_{\frac12-i\infty}^{\frac12+i\infty}z^{-s}\M_{{\alpha}}\left(s\right)ds.\]
Hence an application of Cauchy's integral theorem yields
\begin{equation}
\frac{1}{2\pi i}\int_{\frac12-i\infty}^{\frac12+i\infty}z^{-s}\M_{{\alpha}}\left(s\right)ds=\sum_{n=0}^{\infty} \frac{\Gamma(\frac1\alpha)\Gamma(\alpha n+1)}{\Gamma(n+\frac1\alpha)( n!)^2} (-1)^n z^{\alpha n}=g_{\alpha}(z)
\end{equation}
where we sum  over the poles $s=-\alpha n, n=0,1\ldots$ of $\Gamma\left(\frac{s}{\alpha}\right)$ with  residues $\frac{\alpha(-1)^n}{n!}$.
This shows that $\M_{g_\alpha}=\M_{\alpha}$.
Since $\alpha\in(1,2)$, we have, from \eqref{eq:asymp_Ma}, that $b\mapsto\M_{{\alpha}}(\frac12+ib)\in \Lr$ and by the Parseval identity for the Mellin transform
we conclude that $g_{\alpha}\in \Lp$.  Finally, by means of a standard application of Fubini theorem, see e.g.~\cite[Section 1.77]{Titchmarsh39}), one shows that, for any $x>0$,
\begin{align*}
\La   g_{\alpha}(x)&=\sum_{n=0}^{\infty} \frac{\Gamma(\frac1\alpha)\Gamma(\alpha n+1)}{\Gamma(n+\frac1\alpha)( n!)^2}\M_{\La}(\alpha n+1) (-1)^n x^{\alpha n}
=\sum_{n=0}^{\infty}(-1)^n\frac{x^{\alpha n}}{n!}= \e_{\alpha}(x),
\end{align*}
where we used the expression \eqref{eq:defn_La}. This completes the proof of the lemma.
\end{proof}
Next, let us show that $\M_{\La}$ is the Mellin transform of a random variable that we denote by $I_{\alpha}$. To this end, we write, for any $u>0$,
\[\phi_{\alpha}(u)=\frac{\Gamma(\alpha u+1)}{\Gamma(\alpha u+1-\alpha)}\frac{1}{u-1+\frac1\alpha}=\frac{\alpha}{\Gamma(2-\alpha)}+\int_0^{\infty}(1-e^{-uy})\frac{\alpha(\alpha-1)}{\Gamma(2-\alpha)}\frac{e^{-\frac{y}{\alpha}}}{(1-e^{-\frac{y}{\alpha}})^{\alpha}}dy,\]
where the second identity follows after some standard  computation, see e.g.~\cite[(4.2)]{Patie-06c}. As plainly  $\int_0^{\infty}(y\wedge 1)\frac{e^{-\frac{y}{\alpha}}}{(1-e^{-\frac{y}{\alpha}})^{\alpha}}dy<\infty$, we get,  from \cite[Theorem 3.2]{SchillingSongVondracek10},  that $\phi_{\alpha}$ is a Bernstein function, whose definition is given in \cite[Definition 3.2]{SchillingSongVondracek10}. Moreover, by \cite[Section 5]{SchillingSongVondracek10}, $\phi_{\alpha}$ is the Laplace exponent of a subordinator, that is an increasing process with stationary and independent increments, which we denote by  $(\xi_t)_{t\geq 0}$. Next, observing that for any $n\in \N$,
\begin{equation}\label{eq:I_moment}
\M_{\La}(\alpha n+1)=\frac{n!\Gamma(n+\frac1\alpha)}{\Gamma(\frac1\alpha)\Gamma(\alpha n+1)} = \frac{n!}{\prod_{k=1}^n \phi_{\alpha}(k)},
\end{equation}
we deduce, from \cite[Proposition 3.3]{Carmona-Petit-Yor-97}, that  $(\M_{\La}(\alpha n+1))_{n\geq0}$
is the Stieltjes moment sequence of the random variable $\int_0^{\infty}e^{-\xi_t}dt$.  Moreover, observe from its definition \eqref{eq:defn_La} and  applications of the recurrence relation of the gamma function that $\M_{\La}$ satisfies the functional equation, on $s \in \C_+$,
\[\M_{\La}(\alpha s+1) = \frac{s}{\phi_{\alpha}(s)}\M_{\La}(\alpha(s-1)+1),\quad  \M_{\La}(1)=1,\]
hence, by a uniqueness argument developed in \cite[Section 7]{Patie-Savov-15}, we have
\begin{equation} \label{eq:Mell_Ia}
\M_{\La}(s+1)=\E\left[\left(\int_0^{\infty}e^{-\xi_t}dt\right)^{\frac s\alpha}\right]=\E\left[I^s_{\alpha}\right].
 \end{equation}
 Consequently $\M_{\La}(s)$ is the Mellin transform of the variable $I_{\alpha}=\left(\int_0^{\infty}e^{-\xi_s}ds\right)^{\frac1\alpha}$.  Finally,  since the  law of $\int_0^{\infty}e^{-\xi_t}dt$ is known to   be absolutely continuous, see e.g.~\cite[Proposition 7.7]{Patie-Savov-15}, so is the one of $I_{\alpha}$, therefore we conclude that  $\M_{\La}$ is indeed a Markov multiplier, which provides the first claim of Theorem \ref{thm:intertwin}\eqref{it:Ma}.
 Next, the one-to-one property of $\La$ follows from the fact that the mapping $s\mapsto \M_{\La}(s)$ is clearly zero-free on the line $ 1+i\R$. Moreover, writing $\lambda_{\alpha}$ the density of $I_{{\alpha}}$, we have by dominated convergence that for any $f\in\Co$, $\La f \in \Co$ with
$\|\La f\|_{\infty}\leq \|f\|_{\infty}$, that is, $\La\in \mathbf{B}(\Co)$. On the other hand, for $f\in \Lp$, Jensen's inequality and a change of variable yield
\begin{align*}
\|\La f\|^2 = \int_0^{\infty}\E\left[f(xI_{\alpha})\right]^2dx&\leq \int_0^{\infty}\E\left[f^2(xI_{\alpha})\right]dx =\E[I_{{\alpha}}^{-1}] \: \|f\|^2
\end{align*}
where $\E[I_{{\alpha}}^{-1}]=\M_{\La}(0)=\frac{\Gamma(1-\frac1\alpha)}{\Gamma(1+\frac1\alpha)}<\infty$. Hence $\La\in \mathbf{B}(\Lp)$. Moreover, from Lemma~\ref{lemma:La_G_e}, it is easy to conclude that $\La \rmdq g_{\alpha}=\rmdq \e_{\alpha}$ for all $q>0$, where $\rmdq g_{\alpha} \in\Lp$ since $q\rmdq$ is a unitary operator in $\Lp$. Hence, by the well-known result that $\overline{{\mathrm{Span}}}(\rmdq \e_{\alpha})_{q>0}= \Lp$, we have that $\La$ has a dense range in $\Lp$, which completes the proof of Theorem \ref{thm:intertwin}\eqref{it:Ma}.

\subsection{Proofs of Theorem~\ref{thm:intertwin}\eqref{it:int} and \eqref{it:intg}}
We recall that a collection of $\sigma$-finite measures $(\eta_t)_{t>0}$ is called an \textit{entrance law} for the semigroup $P$ if for any $t,s>0$ and any $f\in \Co, \eta_tP_sf=\eta_{t+s}f$
where $\eta_tf=\int_0^{\infty}f(x)\eta_t(dx)$. We also recall from Appendix~\ref{sec:appen} that $G_{\alpha}$  is the $\frac1\alpha$ power of a gamma variable with parameter $\frac1\alpha>0$, that is $\Prob(G_{\alpha}\in dy)=\frac{e^{-y^{\alpha}}}{\Gamma(1+\frac1\alpha)}dy, y>0$. Now we are ready to state the following Lemma.
\begin{lemma} \label{lemma:decompose_e}
$P$ admits an entrance law $(\eta_t)_{t>0}$ defined for any $t>0$  by $\eta_t f = \eta_1 \mathrm{d}_{t}f=\int_0^{\infty}f(ty)\eta_1(dy)$ where $\eta_1(dy)=\lambda_{\Xa}(y)dy$, with $\lambda_{\Xa} \in\Lp$, is the probability measure of a variable $\Xa$. Its  Mellin transform takes the form
\begin{equation} \label{eq:Mellin_T_r}
\M_{\Xa}(s)=\frac{\Gamma(s)}{\Gamma(\frac{s}{\alpha}+1-\frac1\alpha)}, \quad s\in\mathbb{C}_+.
\end{equation}
Moreover, we have the following factorization of the variable $G_{\alpha}$
\begin{equation*}
G_{\alpha} \eqindist \Xa \times I_{\alpha},
\end{equation*}
where $\eqindist$ stands for the identity in distribution  and $\Xa$ is considered independent of $I_{\alpha}$, which we recall was characterized in \eqref{eq:Mell_Ia}.
\end{lemma}
\begin{proof}
First, let us observe from \eqref{eq:Mellin_T_r} that for any $n\geq 0$,
\begin{equation} \label{eq:Moment_T_r}
\M_{\Xa}(\alpha n+1)=\frac{\Gamma(\alpha n+1)}{n!}=\frac{\prod_{k=1}^n \frac{\Gamma(\alpha  k+1)}{\Gamma(\alpha(k-1)+1)}}{n!}=\prod_{k=1}^n\frac{\psi^{\dag}(\alpha k)}{k},
\end{equation}
where we recall from the proof of Proposition~\ref{prop} that $\psi^{\dag}(u)=\frac{\Gamma(u+1)}{\Gamma(u-\alpha+1)}, u>\alpha-1,$
is the Laplace exponent of the killed L\'evy process $\xi^{\dag}$ defined in \eqref{eq:defk}. Then by \cite[Theorem 1]{BD_moments_self_similar}, we deduce that $(\M_{\Xa}(\alpha n+1))_{n\geq 0}$ is the moment sequence of the variable $X_1^{\alpha}$ under $\Prob_0$, for which we used the fact that since $X$ is a $\frac1\alpha$-self-similar process, $X^{\alpha}$ is a 1-self-similar process whose minimal process is associated, through  the Lamperti mapping, to a L\'evy process with Laplace exponent $\psi_{\alpha}(u)=\psi^{\dag}(\alpha u)$.
 Moreover, note from \eqref{eq:Mellin_T_r} that $\M_{\Xa}$ satisfies the functional equation
$\M_{\Xa}(\alpha s+1)=\frac{\psi_{\alpha}( s)}{s}\M_{\Xa}(\alpha(s-1)+1)$ with $\M_{\Xa}(1)=1$, hence by a uniqueness argument, see again \cite[Section 7]{Patie-Savov-15}, we conclude that $\M_{\Xa}(s+1)=\E_0[X_1^{s}]$ is indeed the Mellin transform of  $X_1$ under $\Prob_0$. Using again the Stirling approximation \eqref{eq:asympt_gamma}, we see that $|\M_{\Xa}(\frac12+ib)|\eqpinfy \bo\left(|b|^{\frac{1}{2\alpha}-\frac12}e^{-\frac{\pi}{2}(1-\frac1\alpha)|b|}\right)$, and thus  $b \mapsto \M_{\Xa}(\frac12+ib) \in \Lr$.  Hence by Mellin inversion and Parseval identity, we get that the law of $\Xa$ is absolutely continuous with a density
$\lambda_{\Xa} \in \Lp$. Now,   recalling that $\eta_1(dy)=\lambda_{\Xa}(y)dy$ and for any $t>0$,
$\eta_t f = \eta_1 \mathrm{d}_{t}f$, we get, from \eqref{eq:Moment_T_r} augmented by a moment identification that $(\eta_t)_{t> 0}$ is an entrance law for the semigroup $P$.  Finally, from the expression of $\M_{\La}$ in \eqref{eq:defn_La}, we conclude that for $s\in\C_+$,
\begin{equation*}
\M_{\Xa}(s)\M_{\La}(s)=\frac{\Gamma(s)}{\Gamma(\frac{s}{\alpha}+1-\frac1\alpha)} \frac{\Gamma(\frac{s-1}{\alpha}+1)\Gamma(\frac{s}{\alpha})}{\Gamma(\frac1\alpha)\Gamma(s)}=\frac{\Gamma(\frac{s}{\alpha})}{\Gamma(\frac1\alpha)}=\M_{G_{\alpha}}(s)
\end{equation*}
where we used for the last identity the expression \eqref{eq:Mel_gam}. We complete the proof by invoking the injectivity of the Mellin transform.
\end{proof}
We are now ready to prove the intertwining relation stated in Theorem \ref{thm:intertwin}\eqref{it:int}.
First, since $s\mapsto \M_{\Xa}(s)$ is zero-free on the line $1+i\R$, we again conclude that the  Markov operator $\Lambda_{\Xa}$  associated to the positive variable $\Xa$, i.e.~$\Lambda_{\Xa}f(x)=\int_0^{\infty}f(xy)\lambda_{\Xa}(y)dy$, is injective on $\Co$. This combined with the fact that the law of $G_{\alpha}$ is the entrance law at time $1$ of the semigroup $Q$ and with the factorization of this latter stated in Lemma~\ref{lemma:decompose_e} provide all conditions for the application of  \cite[Proposition 3.2]{Carmona-Petit-Yor-97},  which gives that for any $t\geq 0$ and $f\in\Co$,  the following intertwining relationship between the Feller semigroups $(P^F_t)_{t\geq 0}$ and $(Q^F_t)_{t\geq 0}$,
\begin{equation}\label{eq:intertwin_Feller}
P^F_t \La f=\La Q^F_t f.
\end{equation}
in $\Co$. Futhermore, since $\Co\cap \Lp$ is dense in $\Lp$, we can extend the intertwining identity into $\Lp$ by continuity of the involved operators and complete the proof of Theorem~\ref{thm:intertwin}\eqref{it:int}. Finally, Theorem~\ref{thm:intertwin}\eqref{it:intg} follows directly from \eqref{eq:intertwin} by recalling that ${}^CD_+^{\alpha}$ and $\mathbf{L}$ are the infinitesimal generators of $P$ and $Q$, respectively, where the $\Lp$-domain of $\mathbf{L}$ is given in \eqref{eq:L2domain_bfL}. This concludes the proof of  Theorem~\ref{thm:intertwin}.
\section{Eigenfunctions and upper frames} \label{sec:eigenfucn}
We start by recalling a few definitions concerning the spectrum of linear operators and we refer to \cite[XV.8]{Dunford1971} for a thorough account on these objects.  Let ${\rm{P}} \in \mathbf{B}(\Lp)$. We say that a complex number $\mathfrak{z} \in S({\rm{P}})$, the spectrum of ${\rm{P}}$, if ${\rm{P}}-\mathfrak{z}  I $ does not have an inverse in $\Lp$ with the following three distinctions:
\begin{itemize}
\item $\mathfrak{z} \in S_p({\rm{P}})$, the point spectrum, if $\Ker({\rm{P}}-\mathfrak{z} I)\neq \{ 0\}$. In this case, we say a function $f_{\mathfrak{z}}$ is an eigenfunction for ${\rm{P}}$, associated to the eigenvalue $\mathfrak{z} $, if $f_{\mathfrak{z}} \in \Ker({\rm{P}}-\mathfrak{z} I)$.
\item $\mathfrak{z} \in S_c({\rm{P}})$, the continuous spectrum, if $\Ker({\rm{P}}-\mathfrak{z} I)= \{ 0\}$ and $\overline{\Ran}({\rm{P}}-\mathfrak{z} I )=\Lp$ but $\Ran({\rm{P}}-\mathfrak{z} I)\subsetneq \Lp$.
\item $\mathfrak{z} \in S_r({\rm{P}})$, the residual spectrum, if $\Ker({\rm{P}}-\mathfrak{z} I)= \{ 0\}$ and $\overline{\Ran}({\rm{P}}-\mathfrak{z}  I)\subsetneq\Lp$.
\end{itemize}
Moreover, we also recall from \cite{Antoine_frames} that a collection of functions $(g_q)_{q>0}$ is a \textit{frame} for $\Lp$ if for all $q>0$ $g_q \in \Lp$ and there exists constants $A,B>0$, called the \textit{frame bounds}, such that, for all $f\in \Lp$,
\begin{equation*}
A\|f\|^2\leq \int_0^{\infty}\left\langle f, g_q\right \rangle^2dq \leq B\|f\|^2.
\end{equation*}
 Moreover, we say $(g_q)_{q>0}$ is \textit{upper frame} if it only satisfies the second inequality. Finally, recalling that $\calJ$ was defined in  \eqref{eq:calJr}, we are ready to state the following claims which include the expression along with substantial properties of the set of eigenfunctions of $P_t$.
\begin{theorem} \label{proposition:J_upper_frame}
	\begin{enumerate}
		\item \label{it:eigen}For any $q,t>0$, $\rmdq\calJ$ is an eigenfunction for $P_t$ associated to the eigenvalue $e^{-q^{\alpha}t}$. Consequently, we have $(e^{-q^{\alpha}t})_{q>0} \subseteq S_p(P_t)$.
		\item Let the linear operator $\ha$ be defined for any $f\in\Lp$ by
		\begin{equation}
		\ha f(q)= \int_0^{\infty}f(x)\calJ(qx)dx,\quad q>0,\label{defn:ha}
		\end{equation}
		then $\ha\in \mathbf{B}(\Lp)$ with $|||\ha|||=\sup_{\|f\|=1}\|\ha f\|\leq \frac{\Gamma(1-\frac1\alpha)}{\Gamma(1+\frac1\alpha)}$. Consequently, the collection of functions $(\rmdq\calJ)_{q\geq 0}$ is a dense upper frame for $\Lp$, with upper frame bound $\frac{\Gamma(1-\frac1\alpha)}{\Gamma(1+\frac1\alpha)}$.
		\item For any $k \in \N$,  $\calJ^{(k)}$ admits the following asymptotic expansion for large $x>0$,
		\begin{eqnarray}\label{eq:asymp_calJ}
		\calJ^{(k)}(x)&\approx & \frac{x^{-k-\alpha}}{\pi\Gamma(1+\frac1\alpha)} \sum_{n=0}^{\infty}a_{n,k}\:x^{-\alpha n}
		\end{eqnarray}	
where $a_{n,k}=(-1)^{n+k}\Gamma(\alpha n+\alpha+k)\sin(\pi\alpha(n+1))$ and $\approx
$ means that for any $N \in \N$, $\calJ^{(k)}(x)- \frac{x^{-k-\alpha}}{\pi\Gamma(1+\frac1\alpha)} \sum_{n=0}^{N}a_{n,k}\: x^{-\alpha n} \stackrel{\infty}{=} \so\left(x^{-k-\alpha-\alpha (N+1)}\right).$
	\end{enumerate}
\end{theorem}
\begin{remark}
Note that there is a cut-off in the nature of the spectrum when one considers the family of operators $P$  indexed by $\alpha \in (1,2)$. Indeed, when $\alpha=2$, then $\mathcal{J}_2=J_{2}\notin\Lp$ (see \eqref{eq:Bessel} for the definition of the Bessel-type function $J_2$) and hence, for all $q,t>0$, $e^{-q^2 t}\notin S_p(P_t)$ but instead $e^{-q^2t}\in S_c(P_t)$.
\end{remark}
\begin{proof}
First,  we recall that $\J$, the Bessel-type function, is  defined in \eqref{eq:Bessel} as an holomorphic function on $\C_{(-\infty,0)^c}$.  As $\alpha\in(1,2)$ and $ \J(x) \eqinfy \bo{\left(x^{\frac{\alpha-2}{4}}\right)}$, see e.g.~\cite{Sekeljic_2010}, we get that $\J\in \Co$. Hence, as above, applying  Fubini's theorem, we obtain, for $x>0$, that
\begin{equation} \label{eq:Ljj}
\La \J(x) = \alpha\sum_{n=0}^{\infty}\frac{(e^{i\pi} x^{\alpha})^n}{n!\Gamma(n+\frac1\alpha)}\M_{\La}(\alpha n+1) = \frac{1}{\Gamma(1+\frac1\alpha)}\sum_{n=0}^{\infty}\frac{(e^{i\pi}x^{\alpha})^n}{\Gamma(\alpha n+1)}=\calJ(x),
\end{equation}
which shows, since $\La\in\mathbf{B}(\Co)$, that both $\calJ\in\Co$ and $\rmdq \calJ\in\Co$ for all $q>0$. Thus, we can use  the  relation \eqref{eq:intertwin_Feller} to get, for all $q>0$ and $x\geq0$,
\begin{equation}\label{eq:intertwin_eigen}
P^F_t \rmdq\calJ(x)=P^F_t\La \rmdq\J(x)=\La Q^F_t\rmdq\J(x)=e^{-q^{\alpha}t}\La \rmdq\J(x)=e^{-q^{\alpha}t}\rmdq\calJ(x).
\end{equation}
Next,  proceeding as in the proof of Lemma~\ref{lemma:La_G_e}, we get, for $|\arg(z)|<\left(\frac1\alpha-\frac12\right)\pi$, that
\[ \calJ(z)=\frac{1}{2\pi i}\int_{\frac12-i\infty}^{\frac12+i\infty}z^{-s}\M_{\calJ}(s)ds,\]
where, for $0<\Re(s)<\alpha$,
\begin{equation}\label{eq:Mellin_calJ}
\M_{\calJ}(s)=\frac{\Gamma(1-\frac{s}{\alpha})\Gamma(\frac{s}{\alpha})}{\Gamma(1-s)\Gamma(\frac1\alpha)}.
\end{equation}
Since from \eqref{eq:asympt_gamma}, $\left|\M_{\calJ}\left(\frac12+ib\right)\right| \eqpinfy\bo\left( |b|^{\alpha a-\frac{\alpha}{2}}e^{-\frac{\pi}{2}(\frac2\alpha-1)|b|}\right)$,
 we get, by the Parseval identity,  that  $\calJ\in\Lp$. Moreover, since $P^F$ coincides with its extension $P$ on $\Co\cap \Lp$, we conclude from \eqref{eq:intertwin_eigen} that, for all $q,t>0$, $\rmdq\calJ$ is an eigenfunction of $P_t$ with eigenvalue $e^{-q^{\alpha}t}$. This completes the proof of the first item. 
 Next, with  $\adjLambda \in \B{\Lp}$ as the $\Lp$-adjoint of $\La \in \B{\Lp}$,  we have, for any $f\in \Lp$,
 \begin{eqnarray*}
 \|\ha f\|^2&=& \int_0^{\infty}\left\langle f,\La \rmdq\J \right \rangle^2 dq= \| H_{\alpha} \adjLambda f\|^2 =\|\adjLambda f\|^2\leq \frac{\Gamma^2(1-\frac1\alpha)}{\Gamma^2(1+\frac1\alpha)}\|f\|^2
 \end{eqnarray*}
 where we  used successively the identity  \eqref{eq:Ljj}, the definition as well as the unitary property of $ H_{\alpha}$, see Proposition~\ref{proposition:Parseval_J}, and the identity $|||\adjLambda|||=|||\La|||\leq \frac{\Gamma(1-\frac1\alpha)}{\Gamma(1+\frac1\alpha)}$, giving the second claim. Furthermore, for any $k\in \N$, by a classical argument since the first series below is easily checked to be uniformly convergent in $z\in\C_{(-\infty,0)^c}$, we get
\begin{equation*}
z^{k}\Gamma(1+\frac1\alpha)\calJ^{(k)}(z)=\sum_{n=0}^{\infty}\frac{(e^{i\pi}z^{\alpha})^n}{\Gamma(\alpha n-k+1)}={}_1\Psi_1(e^{i\pi}z^{\alpha})\approx z^{-\alpha}\sum_{n=0}^{\infty}\frac{(-1)^n}{n!}\frac{\Gamma\left(n+1\right)z^{-\alpha n}}{\Gamma(-\alpha n-\alpha-k+1)},
 \end{equation*}
where, for $|\arg(z)|<\frac{\pi}{2}(2-\alpha)$, we used \cite[Theorem 1]{Paris_Exponentially_Small_Expansion}, with the notation therein, that is, ${}_1\Psi_1$ stands for the Wright function and we made  the choice of parameters $p=q=1,\alpha_1=1, a_1=1,\beta_1=\alpha,b_1=-k+1, \kappa=1+\beta_1-\alpha_1=\alpha \in(1,2)$. The proof is completed by an application of the reflection formula for the gamma function.
\end{proof}

\section{Co-residual functions} \label{sec:residual_fucn}
In this section, we focus on characterizing the spectrum of the $\Lp$-adjoint operator $\hatP=(\hatP_t)_{t\geq 0}$. We point out that the non-self-adjointness of $P_t$ does not ensure the existence of eigenfunctions for $\hatP_t$. In fact, we shall show in the following Lemma that the point spectrum of $\hatP_t$ is empty and $S_p(P_t)=S_r(\hatP_t)$, the residual spectrum of $\hatP_t$.
\begin{lemma}
For each $t\geq 0$, $S_p(\hatP_t)=\emptyset$ and $ (e^{-q^{\alpha}t})_{q> 0} \subseteq S_r(\hatP_t)=S_p(P_t)$.
\end{lemma}
\begin{proof}
Assume that there exists $\mathfrak{z}\in S_p(\hatP_t)$, then there exists a non-zero function $f_{\mathfrak{z}}\in\Lp$ such that $\hatP_t f_{\mathfrak{z}}=\mathfrak{z} f_{\mathfrak{z}}$. Moreover, since $\La$ has a dense range in $\Lp$, we see that $\Ker(\hatLambda)=\{0\}$ and therefore $g_{\mathfrak{z}}=\hatLambda f_{\mathfrak{z}}\neq 0$ with $g_{\mathfrak{z}}\in\Lp$ as $\hatLambda\in\B{\Lp}$. Now by the adjoint intertwining relation of \eqref{eq:intertwin}, we have
\[ Q_t g_{\mathfrak{z}}=Q_t\hatLambda f_{\mathfrak{z}}=\hatLambda \hatP_t f_{\mathfrak{z}} =\mathfrak{z}\hatLambda f_{\mathfrak{z}}=\mathfrak{z} g_{\mathfrak{z}},\]
which implies that $\mathfrak{z}\in S_p(Q_t)$, a contradiction to the fact that $S_p(Q_t)=\emptyset$. Therefore we have $S_p(\hatP_t)=\emptyset$ and moreover, from the known fact that $S_r(\hatP_t)\cup S_p(\hatP_t)=S_p(P_t)$, we conclude that $(e^{-q^{\alpha}t})_{q> 0}\subseteq S_r(\hatP_t)$.
\end{proof}
Next, we will characterize a sequence of the so-called residual functions associated to $S_r(\hatP_t)$, by means of (weak) Fourier kernels. To this end, we first recall from \cite{Patie_Savov_self_similar} that a linear operator $\widehat{\mathcal{H}}$ is called a weak Fourier kernel if there exists a linear space $\D(\widehat{\mathcal{H}})$ dense in  $\Lp$ and $\M_{\widehat{\mathcal{H}}}:\frac12+i\R\rightarrow \C$ such that, for any $f\in\D(\widehat{\mathcal{H}})$,
\begin{equation} \label{cond:weak_Fourier}
b\mapsto \M_{\widehat{\mathcal{H}} f}\left(\frac12+ib\right)=\M_{\widehat{\mathcal{H}}}\left(\frac12+ib\right)\M_{ f}\left(\frac12-ib\right) \in \Lr.
\end{equation}
\begin{theorem} \label{thm:coeigen}
Let us write, for $s\in\C_{(0,1)}$,
\begin{eqnarray} \label{eq:Mellin_hha}
\M_{\hha}(s) &=& \frac{\Gamma\left(\frac1\alpha\right)\Gamma(s)}{\Gamma\left(\frac{1-s}{\alpha}\right)\Gamma\left(1-\frac1\alpha+\frac{s}{\alpha}\right)}. \label{defn:hha}
\end{eqnarray}
Then the following statements hold.
\begin{enumerate}
\item \label{it:hha_inverse} $\hha$ is a weak Fourier kernel and $\D \subseteq \D(\hha)$, where the linear space $\D$ is defined by
\begin{equation} \label{eq:asymp_Mf}
\D=\left\{f\in\Lp;\: \left|\M_f\left(\frac12+ib\right)\right|\eqpinfy\bo\left(|b|^{-\frac12-\epsilon}e^{\frac{-(2-\alpha)\pi}{2\alpha}|b|}\right)\textnormal{ for some }\epsilon>0 \right\}.
\end{equation}
Moreover, we have $\hha \ha f= f$ on $\Lp$ and $\ha \hha g = g$ on $\D(\hha)$. Consequently, $\hha$ is a self-adjoint operator on $\D(\hha)$.
\item \label{it:hha_range_La} We have $\Ran{(\La)}\subseteq \D(\hha)$, and,  on $\Lp$, $\hha \La  = H_{\alpha}$ and $\ha   = \La H_{\alpha}$.
\item \label{it:residual_fcn}
For any $f\in\Lp$ and $t,q>0$, we have $\hha P_t\La f(q)=e^{-q^{\alpha}t}\hha \La f(q)$. Moreover,  for any $1<\kappa<\frac{\alpha}{2-\alpha}$, $\mathcal{E}_{\alp}={\mathrm{Span}}(\e_{\alp,\tau})_{\tau> 0}$ is a dense subset of $\Lp$, and, for all $f\in\mathcal{E}_{\alp}$, we have the integral representation, for almost every (a.e.) $q>0$,
\begin{equation}
\hha f(q)=\int_0^{\infty}f(x)\hatJ(qx)dx
\end{equation}
where, for $|\arg(z)|<\pi$, we set, with $\pa=\frac{\pi}{\alpha}$,
\begin{eqnarray}
 \hatJ(z) &=&\frac{\Gamma\left(\frac1\alpha\right)}{\pi}\sin\left(\pa-z\sin\left(\pa\right)\right)e^{-z\cos(\pa)}. \label{eq:hatJr}
\end{eqnarray}
We say that $\rmdq \hatJ$ is a residual function for $\hatP_t$ (or co-residual function for $P_t$) associated to residual spectrum value $e^{-q^{\alpha}t}$.
\end{enumerate}
\end{theorem}
\begin{proof}
First, since from \eqref{eq:asympt_gamma} and $0<a<1$ fixed,
$\left|\M_{\hha}\left(a+ib\right)\right|\eqpinfy\bo\left( |b|^{a-\frac12}e^{\frac{(2-\alpha)\pi}{2\alpha} |b|}\right)$,
we deduce from \eqref{eq:asymp_Mf} and the fact that, for all $b\in \R$, $\left|\M_f\left(\frac12-ib\right)\right|=\left|\M_f\left(\frac12+ib\right)\right|$,  that \begin{equation*}
b\mapsto \left|\M_{\hha}\left(\frac12+ib\right) \M_f\left(\frac12-ib\right)\right|\eqpinfy\bo\left(|b|^{-\frac12-\epsilon}\right)\in\Lr.
\end{equation*}
Therefore $\D\subseteq \D(\hha)$. Next, observing that for any  $1<\kappa<\frac{\alpha}{2-\alpha}$, $\tau>0$ and $s\in\C_{(0,\infty)}$,
\begin{equation} \label{eq:Mellin_e_taualpha}
\M_{\e_{\alp,\tau}}(s) = \int_0^{\infty}x^{s-1}e^{-\tau x^{\alp}}dx = \tau^{-\frac{s}{\alp}}\alp^{-1}\Gamma\left(\frac{s}{\alp}\right),
\end{equation}
we get
$\left|\M_{\e_{\alp,\tau}}\left(\frac12-ib\right)\right|\eqpinfy\bo\left(|b|^{\frac{1}{2\alp}-\frac12}e^{-\frac{\pi}{2\alp}|b|}\right)$,
and thus, for any  $1<\kappa<\frac{\alpha}{2-\alpha}$, $(\e_{\alp,\tau})_{\tau>0}\subseteq \D$. Moreover, since $(\e_{\alp,\tau})_{\tau>0}$ is dense in $\Lp$, we obtain  that $\D(\hha)$ is dense in $\Lp$ and therefore $\hha$ is a weak Fourier kernel. Next, using the definition of $\ha f$ in \eqref{defn:ha}, we get, by performing a change of variable in \eqref{eq:mm}, that for any $f\in\Lp$ and $s\in \frac12+i\R$,
\begin{equation}\label{eq:Mellin_haf}
\M_{\ha f}(s)=\M_{\calJ}(s)\M_f(1-s)=\frac{\Gamma\left(1-\frac{s}{\alpha}\right)\Gamma\left(\frac{s}{\alpha}\right)}{\Gamma\left(\frac1\alpha\right)\Gamma(1-s)}\M_f(1-s).
\end{equation}
Thus, for such $s$, we have
\[\M_{\hha}(s)\M_{\ha f}(1-s)=\frac{\Gamma\left(\frac1\alpha\right)\Gamma(s)}{\Gamma\left(\frac{1-s}{\alpha}\right)\Gamma\left(1-\frac1\alpha+\frac{s}{\alpha}\right)}\M_{\ha f}(1-s)=\M_f(s).\]
Therefore, an application of the Parseval identity yields that $\ha f\in\D(\hha)$ and  $\hha \ha  f=f$ for all $f\in \Lp$. Similarly, one gets that $\ha\hha g=g$ for all $g\in\D(\hha)$. Next, from \eqref{defn:ha} one gets readily that $\ha$ is self-adjoint in $\Lp$, hence $\hha$ is also self-adjoint as the inverse operator of $\ha$, which concludes the proof of Theorem~\ref{thm:coeigen}\eqref{it:hha_inverse}. Next, from \eqref{eq:mm}, we have, for any $f\in\Lp$,  $\M_{\La f}(s)=\M_{f}(s)\M_{\La}(1-s)$, therefore for at least $s\in \frac12+i\R$,
\begin{eqnarray*}
\M_{\hha }(s) \M_{\La f}(1-s)&=&\frac{\Gamma\left(\frac1\alpha\right)\Gamma(s)}{\Gamma\left(\frac{1-s}{\alpha}\right)\Gamma\left(1-\frac1\alpha+\frac{s}{\alpha}\right)} \M_{\La}(s)\M_{f}(1-s)\\&=&\frac{\Gamma(\frac{s}{\alpha})}{\Gamma(\frac{1-s}{\alpha})}\M_{f}(1-s)=\M_{H_{\alpha}f}(s)
\end{eqnarray*}
where the Mellin transform  of $H_{\alpha}f$  is given in \eqref{eq:Mellin_Hankel}. Since from Proposition~\ref{proposition:Parseval_J}, $H_{\alpha}f \in\Lp$, we get, by the Parseval identity, that $\La f\in\D(\hha)$ and  $\hha \La f = H_{\alpha}f$ for any $f\in\Lp$. Combine this relation with the self-inverse property of $H_{\alpha}$ from Proposition~\ref{proposition:Parseval_J}, we have, for any $f\in\Lp$, that $\hha \La H_{\alpha} f = H_{\alpha}H_{\alpha}f=f$, which implies $\ha =\La H_{\alpha}$ and finishes the proof of Theorem~\ref{thm:coeigen}\eqref{it:hha_range_La}. Next, combining the intertwining relation \eqref{eq:intertwin} and the spectral expansion \eqref{eq:spectral_expansion_Q_t} for $Q_t$, we get that, for any $f\in\Lp, t>0$,
 \begin{eqnarray}
P_t\La f = \La Q_t f =   \La H_{\alpha} \e_{\alpha,t}H_{\alpha}f = \ha \e_{\alpha,t}H_{\alpha}f. \label{Fubini_before}
 \end{eqnarray}
Hence, by observing that $P_t\La f \in \Ran{(\ha)}$ and by means of Theorem~\ref{thm:coeigen}\eqref{it:hha_inverse}, we get
\[\hha P_t \La f =\hha \ha \e_{\alpha,t}H_{\alpha}f =\e_{\alpha,t}H_{\alpha}f=\e_{\alpha,t}\hha \La f.\]
Finally, since, from above, we have,  for any  $1<\kappa<\frac{\alpha}{2-\alpha}$, $(\e_{\alp,\tau})_{\tau>0}\subseteq \D$, we get that $ \hha \e_{\alp,\tau}\in\Lp$ and by combining  \eqref{eq:Mellin_hha} with \eqref{eq:Mellin_e_taualpha}, that, for at least $ s\in \frac12+i\R$,
\[\M_{\hha \e_{\alp,\tau}}\left(s\right)=\frac{\tau^{\frac{s-1}{\alp}}\Gamma(\frac1\alpha)\Gamma(s)\Gamma(\frac{1-s}{\alp})}{\alp\Gamma(1-\frac{1}{\alpha}+\frac{s}{\alpha})\Gamma(\frac{1-s}{\alpha})}.\]
By following  a  line of reasoning similar to the one used in the proof of Lemma~\ref{lemma:La_G_e}, we obtain
\begin{equation*}
\hha \e_{\alp,\tau}(q)=\frac{\Gamma(\frac1\alpha)}{\alp}\sum_{n=0}^{\infty}\frac{(-1)^n\tau^{-\frac{n+1}{\alp}}\Gamma(\frac{n+1}{\alp})}{n!\Gamma(1-\frac1\alpha-\frac{n}{\alpha})\Gamma(\frac{n+1}{\alpha})}q^n=\frac{\Gamma\left(\frac1\alpha\right)}{\alp\pi}\sum_{n=0}^{\infty}\frac{(-1)^n\Gamma\left(\frac{n+1}{\alp}\right)\sin\left((n+1)\pi_{\alpha}\right)}{\tau^{\frac{n+1}{\alp}}n!}q^n,
\end{equation*}
which  defines an entire function since, by the Stirling approximation \eqref{eq:saa}, $\frac{\Gamma(\frac{n+1}{\alp})}{n\Gamma(\frac{n}{\alp})}\eqinfy \bo\left(n^{\frac1\alp-1}\right)$ and $\kappa>1$.  On the other hand, since for any $x>0$,
\begin{eqnarray*}
\frac{\pi}{\Gamma\left(\frac1\alpha\right)}\hatJ(x)&=& \cos\left(\pa\right)\Im\left(e^{-xe^{i\pi_{\alpha}}}\right)+\sin\left(\pa\right)\Re\left(e^{-xe^{i\pi_{\alpha}}}\right)= \sum_{n=0}^{\infty}\frac{(-1)^n}{n!}\sin\left((n+1)\pi_{\alpha}\right) x^{n},
\end{eqnarray*}
a standard application of Fubini's theorem, see again \cite[Section 1.77]{Titchmarsh39}, yields
\begin{eqnarray*}
\int_0^{\infty}\e_{\alp,\tau}(x)\hatJ(qx)dx&=& \frac{\Gamma\left(\frac1\alpha\right)}{\pi}\sum_{n=0}^{\infty}\frac{(-1)^n}{n!}\sin\left((n+1)\pi_{\alpha}\right) q^{n}\int_0^{\infty}e^{-\tau x^{\alp}}x^ndx\\
&=&\frac{\Gamma\left(\frac1\alpha\right)}{\alp\pi}\sum_{n=0}^{\infty}\frac{(-1)^n\Gamma\left(\frac{n+1}{\alp}\right)\sin\left((n+1)\pi_{\alpha}\right)}{\tau^{\frac{n+1}{\alp}}n!}q^n=\hha \e_{\alp,\tau}(q),
\end{eqnarray*}
from which we conclude that $\hha f(q)=\int_0^{\infty}f(x)\hatJ(qx)dx$ for all $f\in\mathcal{E}_{\alp}$. This completes the proof.
\end{proof}

\section{Spectral representation, heat kernel and smoothness properties} \label{sec:spectral}
We have now all the ingredients for   stating and proving the spectral representation of the semigroups $P$ and $\hatP$ along with  the representation of the heat kernel.
\subsection{Spectral expansions of $P$ and $\widehat{P}$ in Hilbert spaces and the heat kernel}
 \begin{theorem} \label{thm}
 	\begin{enumerate}
 	\item For any $g\in \Lp$ and  $t>0$, we have in $\Lp$
 	 		\begin{equation} \label{eq:spectral_dual}
 	 		\widehat{P}_t g=\hha \e_{\alpha,t}\ha g.
 	 		\end{equation}
 	 \item The heat kernel of $P$ admits the representation
 	  \begin{equation}\label{eq:kernel}
 	 P_t(x,y)=\int_0^{\infty}e^{-q^{\alpha}t}\calJ(qx)\hatJ(qy)dq,
 	  \end{equation}
 	 where the integral is locally uniformly  convergent in $(t,x,y) \in \R^3_+$.
 	\item For any $t>T_{\alpha}$, we have in $\Lp$,
 			\begin{equation} \label{eq:spectral_Bessel_P}
 		 P_t f=\ha \e_{\alpha,t}\hha f
 			\end{equation}
 			 where
 			\begin{enumerate}
 			\item \label{it:ran_La}  if  $f\in  \D{(\hha)}$ then $T_{\alpha}=0$,
 			\item \label{it:e_ag} otherwise if  $f \in {\rm{L}}^2\left(\overline{\e}_{\kappa,\eta}\right)$ for some $\kappa\geq \frac{\alpha}{\alpha-1}$ and $\eta>0$, where we set $\overline{\e}_{\kappa,\eta}(x)= e^{\eta x^{\kappa}},x>0$, then  $T_{\alpha}=\frac{\eta}{\alpha-1}\left(2\frac{\alpha-1}{\alpha \eta}\cos((\alpha+1)\pa)\right)^{\alpha}\mathbb{I}_{\{\kappa(\alpha-1)=\alpha \}}$  and $\hha f(q)=\int_0^{\infty}f(y)\hatJ(qy)dy$.
 			\end{enumerate}
 		
 	\end{enumerate}
 \end{theorem}
\begin{remark} \label{remark:Range_Lambda}
    We mention that $\D{(\hha)}\backslash \Leag \neq \emptyset$ meaning that  the two conditions \eqref{it:ran_La} and \eqref{it:e_ag} are applicable under different situations. For instance, for  $0<\beta<\min\left(\frac{\alpha}{2-\alpha},\frac{\alpha}{\alpha-1}\right)$, $\e_{\beta} \in \Ran(\Lambda_{\alpha})\backslash \Leag$, as one can show that $ \Lambda_{\alpha} B_{\beta} = \e_{\beta}$ with
      $ x \mapsto B_{\beta}(x)=\sum_{n=0}^{\infty}\frac{(-1)^n}{\M_{\Lambda_{\alpha}}(\beta n+1)n!}x^{\beta n} \in \Lp$.
    \end{remark}
\begin{proof}
First, since  for any $f\in\Lp$, $\widehat{P}_t f\in\Lp$, we get, for all $q>0$,
\[\ha \widehat{P}_tf(q)=\left\langle \widehat{P}_tf,\rmdq\calJ\right\rangle = \left\langle f,P_t\rmdq\calJ\right\rangle =e^{-q^{\alpha}t}\left\langle f,\rmdq\calJ\right\rangle=e^{-q^{\alpha}t}\calH_{\alpha}f(q)\]
where we used  Theorem~\ref{proposition:J_upper_frame}\eqref{it:eigen} for $\rmdq\calJ \in \Lp$ and  for the third identity.
Therefore, we can apply Theorem~\ref{thm:coeigen}\eqref{it:hha_inverse} to get that for any $g\in\Lp$,
\begin{eqnarray*}
\widehat{P}_tg =\hha \ha \widehat{P}_tg=\hha \e_{\alpha,t}\ha g,
\end{eqnarray*}
which proves  \eqref{eq:spectral_dual}.
On the other hand, for any $f\in\D(\hha),g\in\Lp$, we have, using the self-adjoint property of $\hha$ and $\ha$, see Theorem~\ref{thm:coeigen}\eqref{it:hha_inverse},
\begin{align*}
\left\langle P_tf,g\right\rangle &= \left\langle f,\widehat{P}_t g\right\rangle = \left\langle f,\hha \e_{\alpha,t}\ha g\right\rangle = \left\langle \ha \e_{\alpha,t}\hha f, g\right\rangle,
\end{align*}
which proves \eqref{eq:spectral_Bessel_P} for $f\in  \D{(\hha)}$  and  $T_{\alpha}=0$, that is, the claim  \eqref{it:ran_La}.
 Next, let us consider the density function $\lambda_{\Xa}\in\Lp$ of the random variable $\Xa$, which we recall was studied in Lemma~\ref{lemma:decompose_e}. Then using \eqref{eq:Mellin_haf} again, it is easy to deduce that $\M_{\ha \lambda_{\Xa}}(s)=\frac{\Gamma(\frac{s}{\alpha})}{\Gamma(\frac{1}{\alpha})}$, which coincides with the Mellin transform of $G_{\alpha}$ (see \eqref{eq:Mel_gam}). Hence we have, for all $q>0$, that $\ha \lambda_{\Xa}(q)=\lambda_{G_{\alpha}}(q)=\frac{e^{-q^{\alpha}}}{\Gamma(1+\frac1\alpha)}$. Therefore, we see that for any $\tau>0$,
$q \mapsto \e_{\alpha,t}\ha \rmdtau \lambda_{\Xa}(q) = \frac{e^{-(\tau^{-\alpha}+t)q^{\alpha}}}{\tau \Gamma(1+\frac1\alpha)} \in \mathcal{E}_{\alpha}$.
Hence, using Theorem~\ref{thm:coeigen}\eqref{it:residual_fcn},  we can write
\begin{equation} \label{eq:Pt_integral_kernel}
\hatP_t\rmdtau \lambda_{\Xa} (y)=\hha \e_{\alpha,t}\ha \rmdtau \lambda_{\Xa} (y)=\int_0^{\infty}e^{-q^{\alpha}t}\hatJ(qy)\int_0^{\infty} \lambda_{\Xa}(\tau x)\calJ(qx)dxdq.
\end{equation}
 Next, from \eqref{eq:asymp_calJ} we deduce that $|\calJ(x)|\eqzero \bo(1)$ and $|\calJ(x)|\eqinfy \bo(x^{-\alpha})$ and thus, since $\lambda_{\Xa}$ is a probability density function,
$\int_0^{\infty} \lambda_{\Xa}(\tau x)|\calJ(qx)|dx \leq C(1+q^{-\alpha})$ for some $C=C(\tau)>0$. On the other hand, from \eqref{eq:hatJr}, we get that there exists $\hat{C}>0$ such that for all $y>0$, $|\hatJ(y)|\leq \hat{C} e^{y}$, which justifies an application of  Fubini theorem to obtain
\begin{equation}  \label{eq:fub_h}
\int_0^{\infty}e^{-q^{\alpha}t}\hatJ(qy)\int_0^{\infty} \lambda_{\Xa}(\tau x)\calJ(qx)dxdq = \int_0^{\infty}\lambda_{\Xa} (\tau x)\int_0^{\infty}e^{-q^{\alpha}t}\calJ(qx)\hatJ(qy)dqdx.
\end{equation}
Now let us define the Mellin convolution operator $\calX$  by $\calX f(\tau)=\int_0^{\infty} f(y)\lambda_{\Xa}(\tau y)dy$ and, since   $\lambda_{\Xa}\in\Lp$,  $\calX \in \B{\Lp}$ and by  performing a change of variable in \eqref{eq:mm} we get, from \eqref{eq:Mellin_T_r},  $\M_{\calX}(s)=\M_{\Xa}(s)=\frac{\Gamma(s)}{\Gamma(\frac{s}{\alpha}+1-\frac1\alpha)}$ which is clearly zero-free on $\Re(s)=1$ entailing that $\calX$ is one-to-one in $\Lp$.  Moreover, by means of the same upper bounds used  above, we deduce that for any $y$ fixed,  \[ x\mapsto \int_0^{\infty}e^{-q^{\alpha}t}\calJ(qx)\hatJ(qy)dq \in\Lp\] and thus the right-hand side of \eqref{eq:fub_h} is in $\Lp$ and hence from \eqref{eq:Pt_integral_kernel}, we get that, for any $\tau>0$,
\begin{equation*}
\hatP_t \rmdtau \lambda_{\Xa} (y)=\int_0^{\infty}\lambda_{\Xa} (\tau x)\int_0^{\infty}e^{-q^{\alpha}t}\calJ(qx)\hatJ(qy)dqdx.
\end{equation*}
The one-to-one property of $\calX$  implies that the transition kernel of $\hatP_t$, denoted by $\hatP_t(y,x)$, can be represented, for  a.e.~$y>0$, as $\hatP_t(y,x)=\int_0^{\infty}e^{-q^{\alpha}t}\calJ(qx)\hatJ(qy)dq$. Since  the last integral is also locally uniformly convergent for any $(t,x,y) \in \R^3_+$, and $\hatJ$ is continuous, the identity holds everywhere. This last fact combined with the duality stated in Proposition \ref{prop}\eqref{it:dual} yield the expression \eqref{eq:kernel} by recalling, from Proposition \ref{prop}\eqref{it:dual}, that since the Lebesgue measure serves as reference measure we get that $P_t(x,y)=\hatP_t(y,x), \: t,x,y>0$.  While \eqref{it:ran_La} has been proved above, we now proceed to the justification of \eqref{it:e_ag}.  First,  by the Cauchy-Schwarz inequality, observe that for any  $f\in \Leag$, writing $\widehat{J}_{\overline{\e}}(qy)=\frac{\hatJ(qy)}{\eag(y)}$, we have
\begin{equation*}
\int_0^{\infty}f(y)\hatJ(qy)dy=\int_0^{\infty}f(y)\widehat{J}_{\overline{\e}}(qy)\eag(y)dy \leq ||f||_{\eag} \left|\left|\rmdq\widehat{J}_{\overline{\e}}\right|\right|_{\eag}.
\end{equation*}
Moreover, since for all $y>0$,
$\left|\hatJ(y)\right|\leq Ce^{-y\cos(\pa)}, C>0,$ we have by an application of the Laplace method, see e.g.~\cite[Ex.7.3 p.84]{Olver-74}, that for large $q$,
\begin{eqnarray*}
\left|\left|\rmdq\widehat{J}_{\overline{\e}} \right|\right|_{\eag}^2
\leq	C^2\int_0^{\infty} e^{-\eta y^{\kappa}}e^{-2q\cos(\pa)y}dy \eqinfy {\rm{O}}\left(q^{a}e^{c_{\kappa}q^{\frac{\kappa}{\kappa -1 }}}\right),
\end{eqnarray*}
where $a>0$ and we set $c_{\kappa}= (\kappa-1)\eta^{\frac{1}{1-\kappa}}\left(\frac{2\cos\left((\alpha+1)\pa\right)}{\kappa}\right)^{\frac{\kappa}{\kappa-1}}>0$ since $\kappa>\alpha>1$. Note that $c_{\frac{\alpha}{\alpha-1}}=T_{\alpha}$ and for any $t>T_{\alpha}$, since $\alpha\geq \frac{\kappa}{\kappa-1}$, $q\mapsto F_{\kappa}(q)=\e_{{\alpha},t}(q)\left( C+ q^{a}e^{c_{\kappa}q^{\frac{\kappa}{\kappa -1 }}}\right)$ is integrable on $\R_+$. This justifies an application of Fubini Theorem which gives that, for such $f,t$ and $x>0$,
\begin{equation}
P_tf(x)=\int_0^{\infty}f(y)\int_0^{\infty}e^{-q^{\alpha}y}\calJ(qx)\hatJ(qy)dqdy.
\end{equation}
Finally,  as $F_{\kappa} \in \Lp$ and from Theorem \ref{thm}, the sequence $(\rmdq\calJ)_{q>0}$ is an upper frame, we obtain that
in fact $P_tf \in \Lp$ and, in $\Lp$, $P_t f = \ha \e_{\alpha,t}\hha f$  with $\hha f(q)=\int_0^{\infty}f(y)\hatJ(qy)dy$. This completes the proof.
\end{proof}

 \subsection{Regularity properties}
Finally, we extract from the spectral decomposition stated in Theorem \ref{thm} the  following regularity properties  as well as an alternative representation of the heat kernel.
\begin{theorem} \label{theorem:kernel}
  \begin{enumerate}	
  \item For any $f \in \Leag \cup \D{(\hha)}$, $(t,x)\mapsto P_tf(x) \in \Cd^{\infty}((T_{\alpha},\infty) \times \R_+)$ and $T_{\alpha}$ was defined in Theorem \ref{thm}.
   \item \label{it:sh}  We have $(t,x,y)\mapsto P_t(x,y) \in \Cd^{\infty}(\R_+^3)$ and, for any non-negative integers $k,\frakp,\frakq$,
     		\begin{equation} \label{eq:kernel_integral_form_P}
     		\frac{d^{k}}{dt^{k}}P_t^{(\frakp,\frakq)}(x,y)= (-1)^k\int_0^{\infty}q^{\alpha k}e^{-q^{\alpha}t} \left(\rmdq\calJ\right)^{(\frakp)}(x) \left(\rmdq\hatJ\right)^{(\frakq)}(y)dq
     		\end{equation}
     where the integral is locally uniformly  convergent in $(t,x,y) \in \R^3_+$.
   \item Moreover, the heat kernel can be written in a series form as 	
    	\begin{equation} \label{eq:kernel_series_form_P}
    	P_t(x,y)=\sum_{n=0}^{\infty}(1+t)^{-n-\frac1\alpha}\mathcal{P}_n(x^{\alpha}) \mathcal{V}_n\left( y(1+t)^{-\frac1\alpha}\right),
    	\end{equation}
    	where $\mathcal{P}_n(x) = \frac{1}{\Gamma(1+\frac1\alpha)}\sum_{k=0}^n (-1)^k\frac{{ n \choose k}k!}{\Gamma(\alpha k+1)} x^{k}$  and $\mathcal{V}_n(y)=\frac{1}{n!}\int_0^{\infty}q^{\alpha n}e^{-q^{\alpha}}\hatJ(qy)dq$ and the series is locally uniformly convergent in $(t,x,y)\in \R^3_+$.
  \end{enumerate}
\end{theorem}
\begin{proof}
We actually  prove only the  item \eqref{it:sh} as  the  first item  follows by developing similar arguments.
 First, from Theorem~\ref{proposition:J_upper_frame} and Theorem~\ref{thm:coeigen}, we have that $\calJ, \hatJ \in \Ci$ and for any $x,y>0$ fixed and non-negative integers $k,\frakp,\frakq$,
\begin{equation*}
\left|\frac{d^k}{dt^k}e^{-q^{\alpha}t} (\rmdq\calJ)^{(\frakp)}(x)(\rmdq\hatJ)^{(\frakq)}(y) \right| \eqinfy \bo\left(q^{\alpha(k-1)+\frakq}e^{-q^{\alpha}t + qy}\right)
\end{equation*}
 and
\begin{equation*}
 \left|\frac{d^k}{dt^k}e^{-q^{\alpha}t} (\rmdq\calJ)^{(\frakp)}(x)(\rmdq\hatJ)^{(\frakq)}(y) \right| \eqzero \bo\left(q^{\alpha k+\frakp+\frakq}\right).
 \end{equation*}
  Hence \eqref{eq:kernel} yields
\begin{equation*}
\frac{d^k}{dt^k}P_t^{(\frakp,\frakq)}(x,y)
=(-1)^{k}\int_0^{\infty}q^{\alpha k} e^{-q^{\alpha}t} (\rmdq\calJ)^{(\frakp)}(x)(\rmdq\hatJ)^{(\frakq)}(y) dq
\end{equation*}
where the integral is locally uniformly convergent in $(t,x,y)\in \R^3_+$.
Hence  $(t,x,y)\mapsto P_t(x,y) \in \Cd^{\infty}(\R_+^3)$. To prove \eqref{eq:kernel_series_form_P},  we first observe, from \cite[Proposition 2.1(ii)]{Craven-Csordas-89_Jensen_Turan},  that for any $x, q \in \R_+$,
\begin{equation} \label{eq:calJ_Poly}
e^{q^{\alpha}} \calJ\left(q x\right)= \sum^{\infty}_{n=0} \mathcal{P}_{n}(x^{\alpha}) \frac{q^{\alpha n}}{n!},
\end{equation}
which by substitution in \eqref{eq:kernel} gives, assuming, for a moment, that one may interchange the sum and integral,
  \begin{eqnarray*}
 P_t(x,y)&=& \int_0^{\infty}e^{-q^{\alpha}(t+1)}\hatJ(q y) \sum^{\infty}_{n=0}  \frac{\mathcal{P}_{n}(x^{\alpha})}{n!}  q^{\alpha n}dq
 = \sum_{n=0}^{\infty}(1+t)^{-n-\frac1\alpha} \mathcal{P}_{n}(x^{\alpha})\mathcal{V}_n\left(y(1+t)^{-\frac1\alpha}\right).
 \end{eqnarray*}
In order to justify the interchange we provide some uniform bounds for large $n$ of $\mathcal{P}_{n}$ and $\mathcal{V}_n$. First, since $z\mapsto \calJ(z^{\frac1\alpha})$ is an entire function of  order $\overline{\lim}_{n\rightarrow \infty}\frac{n\ln n}{\Gamma(\alpha n+1)}=\frac1\alpha$ and  type $1$, by following a line of reasoning similar to the proof of \cite[Theorem 8.4(5)]{Patie-Savov-15}, we obtain the following sequence of inequalities, valid for all $x>0$ and $n$ large,
 \begin{eqnarray*}
 |\Poly_n(x)|\leq \Poly_n(-x)&=&\frac{n!}{2\pi i}x^{ n}\oint_{nx}e^{\frac{z}{x}}\calJ\left(-z^{\frac1\alpha}\right)\frac{dz}{z^{n+1}} \\&\leq& e^{n^{\frac1\alpha}x^{\frac1\alpha}}\frac{n!e^{-n\ln n}}{2\pi}\int_0^{2\pi}e^{n\cos \theta}d\theta=\bo{\left(n^{\frac12}e^{(xn)^{\frac1\alpha}}\right)}
 \end{eqnarray*}
 where the contour is a circle centered at 0 with radius $nx>0$ and for  the last inequality we used the bound $n!\leq e^{1-n}n^{n-\frac12}$. Hence, we have, for all fixed $x>0$ and $n$ large,
 \begin{equation}\label{eq:bound_P_n}
 |\Poly_n(x^{\alpha})|=\bo{\left(n^{\frac12}e^{xn^{\frac1\alpha}}\right)}.
  \end{equation} Next, since for any $q>0$, $|\hatJ(q)|\leq \hatJ(-q)\leq C e^q$, for some constant $C=C(\alpha)>0$, we get, for all $y>0$ and $n\in \N$,
 \begin{eqnarray*}
 \left|\mathcal{V}_n(y)\right| &\leq& \frac{C}{n!}\int_0^{\infty}e^{-q^{\alpha}}q^{\alpha n}\sum_{k=0}^{\infty}\frac{(yq)^k}{k!}d q= C\sum_{k=0}^{\infty}\frac{\Gamma(n+1+\frac{k}{\alpha})}{n!k!}y^{k}
 \end{eqnarray*}
 where for the equality we use the integral representation of the gamma function. Now,  by performing the same computations that in the proof of \cite[Proposition 2.2]{Patie2015}, we get, for all $y>0$ and $n$ large,
 \begin{equation} \label{eq:bound_W_n}
 \left|\mathcal{V}_n(y)\right| = \bo{\left(ne^{\bar{c}_{\alpha}y n^{\frac1\alpha}}\right)}
 \end{equation}
 where $\bar{c}^{\alpha}_{\alpha}=\frac{\alpha}{\alpha-1}$. Hence combining the bounds \eqref{eq:bound_P_n} and \eqref{eq:bound_W_n}, we obtain, for any fixed $x,y,t>0$ and large $n$,
 \begin{equation} \label{eq:asymp_series}
 (1+t)^{-n-\frac1\alpha}\left|\Poly_n(x^{\alpha})\mathcal{V}_n\left(y(1+t)^{-\frac1\alpha}\right)\right| = \bo{\left(n^{\frac32} e^{(\bar{c}_{\alpha}y(1+t)^{-\frac1\alpha}+x) n^{\frac1\alpha} -\ln(1+t)n}\right)},
 \end{equation}
which justifies the interchange and completes the proof.
\end{proof}

\appendix
\section{The $\alpha$-Bessel semigroup and the operator $H_{\alpha}$} \label{sec:appen}
We say $Q=(Q_t)_{t\geq 0}$ is an $\alpha$-Bessel semigroup with index $1<\alpha<2$ if it is a Feller semigroup  whose  infinitesimal generator is given by
\begin{equation}\label{eq:L_Q}
\mathbf{L}f(x)=\frac{2}{\alpha^2}x^{2-\alpha}f^{(2)}(x)+\frac{2}{\alpha}\left(\frac{2}{\alpha}-1\right)x^{1-\alpha}f^{(1)}(x), \: x>0,
\end{equation}
where $ f \in \D_{\mathbf{L}}=\{f\in \Co; \:\mathbf{L}f \in \Co, f^{+}(0)=0 \}$, the domain of $\mathbf{L}$,
with $f^{+}(x)=\lim_{h\downarrow 0} \frac{f(x+h)-f(x)}{s(x+h)-s(x)}$
is the right-derivative of $f$ with respect to the scale function $s(x)=\frac{x^{\alpha-1}}{\alpha-1}$.
We point out that
\begin{equation} \label{eq:defn_Q}
Q_tf(x)=K_t\mathtt{p}_{\frac{1}{\alpha}}f(x^{\alpha}),\: x>0,
\end{equation}
where $K=(K_t)_{t\geq 0}$ is the semigroup of a squared Bessel process of  dimension $\frac{2}{\alpha}$, or equivalently of order $\frac{1}{\alpha}-1$ and $\mathtt{p}_{\frac{1}{\alpha}} f(x)=f(x^{\frac{1}{\alpha}})$.  We refer in this part to \cite[Appendix 1]{Borodin-Salminen-02} for concise information on  squared Bessel processes that can be easily transferred to $Q$ by means of the identity \eqref{eq:defn_Q}.
Furthermore, writing $(\vartheta_t)_{t>0}$  for the entrance law of $Q$, we have  $\vartheta_t f = \int_0^{\infty} f(ty)\lambda_{G_{\alpha}}(y)dy$ where $\lambda_{G_{\alpha}}(y)=\frac{e^{-y^{\alpha}}}{\Gamma(1+\frac1\alpha)},y>0,$ is the density of the variable  $G_{\alpha}$. Note that $G^{\alpha}_{\alpha}$  is simply a  gamma variable of  parameter $\frac1\alpha$, the law of this latter being the entrance law at time $1$  of $K$. The Mellin transform of $G_{\alpha}$ is given by
\begin{equation} \label{eq:Mel_gam}
\M_{G_{\alpha}}(s)=\frac{\Gamma(\frac{s}{\alpha})}{\Gamma(\frac1\alpha)}, \quad \Re(s)> 0.
\end{equation}
Next, defining the function $\J$, for $z\in\C_{(-\infty,0)^c}$, by
\begin{equation} \label{eq:Bessel}
\J(z)=\alpha \sum_{n=0}^{\infty}\frac{(e^{i\pi}z^{\alpha})^n}{n!\Gamma(n+\frac1\alpha)},
\end{equation}
we can deduce  from \cite[Appendix 1]{Borodin-Salminen-02} that for any $q, t,x\geq 0$,
\[Q_t \rmdq \J(x) = e^{-q^{\alpha}t}\rmdq\J(x).\]
Next, we introduce the linear operator defined, for a smooth function $f$ on $q>0$,  by
\begin{equation} \label{eq:defn_Hankel}
H_{\alpha} f(q)=\int_0^{\infty} \J(qx)f(x)dx.
\end{equation}
Then, $H_{\alpha}$ has the following properties reminiscent of the classical Hankel transform.
\begin{proposition} \label{proposition:Parseval_J}
$ H_{\alpha}$ is a unitary and self-inverse operator on $\Lp$, i.e.~$\|H_{\alpha}f\|=\|f\|$ and $H_{\alpha} H_{\alpha}f=f$ for all $f\in\Lp$. Moreover, for any $f\in \Lp$, the Mellin transform of $H_{\alpha}f$ is given by
\begin{equation}\label{eq:Mellin_Hankel}
\M_{H_{\alpha}f}(s)=\M_{\J}(s)\M_f(1-s)=\frac{\Gamma(\frac{s}{\alpha})}{\Gamma(\frac{1-s}{\alpha})}\M_f(1-s), \quad s \in \C_{(0,1)}.
\end{equation}
\end{proposition}
\begin{proof}
First, note that $\J(x)=\alpha x^{\frac{\alpha-1}{2}}\mathtt{J}_{\frac1\alpha-1}\left(2x^{\frac{\alpha}{2}}\right)$ where $\mathtt{J}_{\frac1\alpha-1}$ denotes the standard Bessel function of the first kind of order $\frac1\alpha-1$, see e.g.~\cite[Section 5.3]{Lebedev-72}. Then recall that the standard Hankel transform is defined, for any $g\in \Lm$ where $m(dx)=xdx$, as
\begin{equation*}
\mathtt{H}_{\alpha} g(r)=\int_0^{\infty}\mathtt{J}_{\frac1\alpha-1}(r x)g(x)xdx, \quad r>0.
\end{equation*}
Then by \cite[Chapter 9]{Poularikas_96}, $\mathtt{H}_{\alpha}$ is unitary and self-inverse on $\Lm$, i.e.~for any $g\in \Lm$, we have $\|\mathtt{H}_{\alpha} g\|_{m}=\|g\|_{m}$ and $\mathtt{H}_{\alpha}\mathtt{H}_{\alpha}g =g$. Now for any $f\in\Lp$, we set $g(x)=x^{\frac1\alpha-1}f\left((\frac{x}{2})^{\frac{2}{\alpha}}\right)$. Then it can be easily checked, through a standard change of variable, that $g\in \Lm$ and $\|g\|^2_{m}=\alpha 2^{\frac2\alpha-1}\|f\|^2$. Therefore, by applying a change of variable, one gets
\begin{eqnarray*}
\|H_{\alpha}f\|^2&=&\int_0^{\infty}\left| \int_0^{\infty}f(x) \J(qx)dx \right|^2 dq
= \frac{1}{ 2^{\frac2\alpha}}\int_0^{\infty}\left|\int_0^{\infty}\mathtt{J}_{\frac1\alpha-1}(q^{\frac{\alpha}{2}}y)g(y)ydy\right|^2 \frac{dq}{q^{1-\alpha}}\\
&=& \frac{1}{ 2^{\frac2\alpha}}\int_0^{\infty}q^{\alpha-1}\left|\mathtt{H}_{\alpha} g(q^{\frac{\alpha}{2}})\right|^2dq=\frac{1}{\alpha 2^{\frac2\alpha-1}}\|\mathtt{H}_{\alpha} g\|_{m}^2 
=\frac{1}{\alpha 2^{\frac2\alpha-1}}\|g\|_{m}^2= \|f\|^2.
\end{eqnarray*}
This proves that $H_{\alpha}$ is a unitary operator. Next, for any $f\in\Lp$, again by change of variable, we have
\begin{eqnarray*}
H_{\alpha}H_{\alpha}f(y)=\int_0^{\infty}\J(qy)\int_0^{\infty}f(x)\J(qx)dxdq=\frac{y^{\frac{\alpha-1}{2}}}{2^{\frac1\alpha-1}}\mathtt{H}_{\alpha}\mathtt{H}_{\alpha}g(2y^{\frac{\alpha}{2}})=\frac{y^{\frac{\alpha-1}{2}}}{2^{\frac1\alpha-1}}g(2y^{\frac{\alpha}{2}})=f(y),
\end{eqnarray*}
which proves that $H_{\alpha}$ is self-inverse. Next, using again a change of variable in \eqref{eq:mm}, we have $\M_{H_{\alpha}f}(s)=\M_{\J}(s)\M_f(1-s)$,
where for $0<\Re(s)<1$, $\M_{\J}(s)=\frac{\Gamma(\frac{s}{\alpha})}{\Gamma(\frac{1-s}{\alpha})}$ can be proved by the Mellin-Barnes integral representation of Bessel functions, see e.g.~\cite[Section 3.4.3]{Paris01}, which gives that
\begin{equation*}
\frac{1}{2\pi i}\int_{\frac12-i\infty}^{\frac12+i\infty}z^{-s}\frac{\Gamma(\frac{s}{\alpha})}{\Gamma(\frac{1-s}{\alpha})} ds = \alpha\sum_{n=0}^{\infty}\frac{(e^{i\pi}z^{\alpha})^n}{n!\Gamma(n+\frac1\alpha)}=\J(z).
\end{equation*}
This concludes the proof of the Proposition.
\end{proof}
Next, by referring to \cite[Chapter II]{Borodin-Salminen-02}, we see that the speed measure of $Q$ is (up to a multiplicative positive constant) the Lebesgue measure, hence $Q$ extends uniquely to a self-adjoint contractive $\mathcal{C}_0$-semigroup on $\Lp$, also denoted by $Q$ when there is no confusion (otherwise, we may denote $Q^F$ for the Feller semigroup). The infinitesimal generator $\mathbf{L}$ of this $\Lp$-extension is an unbounded self-adjoint operator on $\Lp$, and, by \cite[Remark 3.1]{McKean_56}, its $\Lp$-domain, denoted by $\D_{\mathbf{L}}(\Lp)$, is given by
\begin{equation} \label{eq:L2domain_bfL}
\D_{\mathbf{L}}(\Lp) = \{f\in\Lp;\: \mathbf{L}f\in\Lp, f^{+}(0)=0\}.
\end{equation}
Moreover, for any $t\geq 0$, $Q_t \in \B{\Lp}$ with $S(Q_t)=S_c(Q_t)=(e^{-q^{\alpha}t})_{q\geq 0}$ and $S_p(Q_t)=S_r(Q_t)=\emptyset$.  Finally, using the spectral expansion of the self-adjoint squared Bessel operator $K_t$, see e.g.~\cite[Section 6]{NS_Lp_contrac_Laguerre} and \cite{Muckenhoupt_Stein_65}, one can deduce that for any $t>0$ and $f\in \Lp$, $Q_tf$ has the following spectral expansion in $\Lp$,
\begin{eqnarray}\label{eq:spectral_expansion_Q_t}
Q_t f =  H_{\alpha} \e_{\alpha,t}H_{\alpha}f.
\end{eqnarray}
\bibliographystyle{plain}

\def\cprime{$'$}

\end{document}